\documentclass[12pt, a4paper]{amsart}
\usepackage[english]{babel}
\usepackage{amsmath,enumerate, amsfonts, amssymb,amsthm, xypic}
\usepackage[dvips]{graphics} 
\usepackage{color}
\usepackage{graphicx}
\usepackage{hyperref}
\usepackage[hyperpageref]{backref}

\newtheorem{thm}{Theorem}[section]
\newtheorem{lemma}[thm]{Lemma}
\newtheorem{prop}[thm]{Proposition}
\newtheorem{cor}[thm]{Corollary}

\theoremstyle{definition}
\newtheorem{defi}[thm]{Definition}

\newtheorem{ex}[thm]{Example}
\newtheorem{rmk}[thm]{Remark}

\DeclareMathOperator{\R}{\mathbf R}
\DeclareMathOperator{\C}{\mathbf C}
\DeclareMathOperator{\N}{\mathbf N}
\DeclareMathOperator{\Z}{\mathbf Z}

\DeclareMathOperator{\id}{id}

\DeclareMathOperator{\Int}{Int}

\DeclareMathOperator{\Ima}{Im}

\DeclareMathOperator{\st}{st}
\DeclareMathOperator{\lk}{lk}
\DeclareMathOperator{\eps}{\epsilon}

\DeclareMathOperator{\wR}{\tilde \R}

\title{Real Milnor Fibres and Puiseux Series}
\author{Goulwen Fichou and Masahiro Shiota}
\thanks{We thank M. Takase for his advice on knots}
\address{IRMAR (UMR 6625), Universit\'e de Rennes 1, Campus de
  Beaulieu, 35042 Rennes Cedex, France}
\address{Graduate School of Mathematics, Nagoya University, Chikusa, Nagoya, 
464-8602, Japan}
\date\today
\subjclass[2010]{14P15}

\begin{document}
\begin{abstract} Given a real polynomial function and a point in its zero locus, we defined a set consisting of algebraic real Puiseux series naturally attached to these data. We prove that this set determines the topology and the geometry of the real Milnor fibre of the function at this point. To achieve this goal, we balance between the tameness properties of this set of Puiseux series, considered as a real algebraic object over the field of algebraic Puiseux series, and its behaviour as an infinite dimensional object over the real numbers.
\end{abstract}
\maketitle


Let $\R$ denote the field of real numbers, and $\N$ denote the non-negative integers. Denote by $\wR$ the field of continuous semialgebraic curve germs $\gamma : (0,\epsilon) \to \R$, which we identify with the field $\wR=\R_{alg} ((t^{\mathbb Q}))$  of algebraic Puiseux series over $\R$ (cf. \cite{BCR}). Recall that the subring of algebraic Puiseux series of the form $\sum_{i\in\N}a_i t^{i/p}$ is carried to the subring of continuous semialgebraic curve germs $[0,\,\infty)\to\R$ at 0.

Let $f$ be a polynomial function on $\R^n$, with $n\in \N^*$, and denote by $\tilde f :\wR^n \to \wR$ the extension of $f$ defined by $\tilde f (\gamma (t))=f\circ \gamma (t)$ for $\gamma \in \wR$ an algebraic Puiseux series. Let $x_0\in \R^n$ be a vanishing point for $f$. The object of study of the present paper is the subset $\mathcal F_{f,x_0} \subset \wR^n$ of continuous semialgebraic curve germs $\gamma :(\R,0) \to (\R^n,x_0)$ such that $f\circ \gamma (t)=t$, namely
$$\mathcal F_{f,x_0}=\{\gamma \in \wR^n:~~ \gamma (0)=x_0,~~ f\circ \gamma (t) =t\}.$$
We aim to relate the topology and the geometry of $\mathcal F_{f,x_0}$ to the topology and the geometry of the Milnor fibre associated with the real polynomial function $f$ at $x_0$, which can be described as the semialgebraic set
$$F_{r,a}(f)=\{x\in \R^n: ~~|x-x_0|<r,~~f(x)=a\}$$
for $0<a \ll r\ll 1$.

\vskip 5mm

The local study of the singular points of complex hypersurfaces has a rich story initiated by J. Milnor in its fundamental book \cite{FMilnor}, where he established the so-called Milnor fibration Theorem. More recently, the development of Motivic Integration \cite{DL} has brought a new enlightenment  on the subject, via the motivic zeta function introduced by J. Denef and F. Loeser, together with the analytic Milnor fibre defined by J. Nicaise and J. Sebag \cite{NS}. Even more recently, E. Hrushovski and F. Loeser have established a direct connection between the analytic Milnor fibre and the topological Milnor fibre, together with the motivic Milnor fibre \cite{HL}, passing through the integration into valued fields developped by E. Hrushovski and D. Kazhdan \cite{HK}.

In the real context, the action of the monodromy operator on the Milnor fibre disappears, which gives rise to the notion of positive and negative Milnor fibres, as study for example by C. McCrory and A. Parusi\'nski \cite{MCP}. However the motivic counterpart, initiated by G. Comte and G. Fichou \cite{CF} and Y. Yin \cite{Y}, does not provide a full understanding of the global feature. Our motivation to consider the set $\mathcal F_{f,x_0}$ is to study a naive real version of the analytic Milnor fibre. Looking at points, we obtain a set of real algebraic Puiseux series, which can almost (or more precisely weakly, see below) be considered as a classical semialgebraic object in real algebraic geometry, replacing the field of real numbers by the real closed field of algebraic Puiseux series over $\R$. 
Note that, if the definition of $\mathcal F_{f,x_0}$ makes sense for any continuous semialgebraic function $f$,
however $\mathcal F_{f,x_0}$
is not necessarily a semialgebraic subset of $\wR^n$ due to the condition that the arcs have origin in $x_0\in \R^n$. This condition can be described by a valuation condition, saying that we consider those arcs with strictly positive valuation after the translation by $\gamma \mapsto \gamma-x_0$. Such sets are sometimes called $T$-convex, or weakly o-minimal \cite{MMS}.

\vskip 5mm

Our aim in this paper is to show that the topology and the geometry of $\mathcal F_{f,x_0}$, which is definitively a natural and intrinsic object in real geometry, determine the topology and the geometry of the (positive) real Milnor fibre of $f$ at $x_0$, which can be considered as a semialgebraic set, but only well-defined up to the choice of (sufficiently small) constants. Note that a similar result holds for the negative Milnor fibre, a negative sign in front of $t$ being necessary in the definition of $\mathcal F_{f,x_0}$ in that case. We prove also that some natural homology groups of $\mathcal F_{f,x_0}$ coincide with the classical homology groups of the (positive) real Milnor fibre. The main achievement of the paper states that $\mathcal F_{f,x_0}$ completely determines the (positive) real Milnor fibre up to semialgebraic homeomorphism (cf Theorem \ref{main}.(1)), and determines it up to Nash diffeomorphism in dimension different from 5 and 6 (cf Theorem \ref{main}.(2)).

\vskip 5mm

As a real closed field, $\wR$ is naturally equipped with the ordered topology which coincides with the $t$-adic topology. But we may also regard $\wR$ as a $\R$-vector subspace of $\R^{\mathbb Q}=\prod_{r\in \mathbb Q}\R t^r$, and consider $\R^{\mathbb Q}$ as a topological space with the product topology. Then $\wR$ may also be equipped with the induced topology.

In the paper, we balance between the tameness properties of $\mathcal F_{f,x_0}$ as a subset of $\wR^n$ close to be semialgebraic, and its behaviour as a topological set of infinite dimension in $\R^{\mathbb Q}$. This situation leads to the study in the first part of the paper of the notion of weak continuity, and its relationship with the continuity of semialgebraic maps defined over $\wR$. We discuss in part 2 the associated homology theories, preparing the material for the comparison of homologies given in the third part as Theorem \ref{thm-hom}. In the fourth part, we focus on the semialgebraic characterisation of the real Milnor fibre, using notion form piecewise linear topology \cite{PL}. The last part is dedicated to the Nash characterisation, using the theory of topological microbundles of J. Milnor \cite{Milnor}. 

Note that all the results of the paper works verbatim for a Nash function in place of a polynomial function $f$, contrary to the analytic case for the reason that a globally subanalytic triviality theorem, analog to the Nash triviality theorem in \cite{CS2}, is not yet available.

\vskip 5mm

To distinguish an interval $(a,\,b)$ in $\R$ and in $\tilde\R$, we write $(a,\,b)$ in $\tilde\R$ as $(a,\,b)_{\tilde\R}$. 
A semialgebraic set is a semialgebraic set over $\R$ and an $\tilde\R$-semialgebraic set is a semialgebraic over $\tilde\R$. 
For a semialgebraic set $X\subset \R^n$, denote by $\tilde {X}$ the set of continuous semialgebraic germs at
$0\in \R$ of continuous semialgebraic functions from $(0;\infty)$ to
$X$. For $x \in X$, we denote $\tilde x\in \tilde X$ the germ of the constant
function equal to $x$. Let $X$ and $Y$ be semialgebraic sets and $h: X \to Y$ be a continuous semialgebraic map. Let $\tilde h:\wR
  \to \wR$ be defined by $\tilde h (\gamma (t))=h \circ \gamma
  (t)$ for $\gamma \in \wR$. We know from \cite{FS} that $\tilde h$ is continuous for the $t$-adic topology, but not necessarily for the product topology.

We define $\mathfrak m_+\subset \wR$ to be the set of infinitely small positive elements in $\wR$, namely:
$$\mathfrak m_+=\{\gamma \in \wR:~~0<\gamma < \tilde x \text{~for all~}x\in (0,+\infty)_{\R}\}.$$



\section{Weak continuity} 

A {\it Nash manifold} is a semialgebraic $C^\infty$ submanifold of some $\R^n$ and a {\it Nash map} between Nash manifolds is a $C^\infty$ map with semialgebraic graph.
The Milnor fibre $F_f(r,a)$ as considered in the introduction is a Nash manifold with boundary, however the set $\mathcal F_{f,x_0}$ is not so, even by changing $\R$ with the field $\tilde\R$ of algebraic Puiseux series over $\R$. We will regard $\mathcal F_f$ as a local Nash manifold, see definition \ref{def-loc} below.
Properties of semialgebraic sets, Nash manifolds and Nash maps are explained in \cite{BCR} and \cite{Shiota}.
We will recall some of them for the convenience of the reader.

  For any ordered field $R$, we give a topology on $R$ by open intervals, and we called it the $R$-topology. We denote by $(a,\,b)_R$ the open interval defined by $a$ and $b$ in $R$, and called it an $R$-interval in order to emphasise again the dependence on $R$ (if $R$ is not the real numbers). 
In the same way as in the real number case, we define $R$-polynomial functions on $R^n$, $R$-algebraic sets in $R^n$, $R$-semialgebraic sets in $R^n$ and $R$-continuous maps from $R^n$ to $R^m$.

If $R$ is a real closed field, for example $\tilde\R$, we call a map $\phi:R\to R$ of {\it class} $R$-$C^1$ if, for every $x_0\in R$, the difference quotient
$$\frac{\phi(x+x_0)-\phi(x_0)}{x}$$ 
converges in $R$ as $x$ tends to $0$ in $R$.
In the same way we define an $R$-$C^k$ map $R^n\to R^m$, an $R$-Nash manifold and an $R$-Nash map.
Note that an $R$-Nash manifold $M$ admits a finite system of $R$-Nash coordinate neighbourhoods of the form $R^n\to M$.

\vskip 5mm

Consider the case $R=\tilde\R$. Then the set $\mathfrak m_+\subset \wR$ of infinitely small positive elements is open but not $\tilde\R$-semialgebraic.

\begin{defi}\label{def-loc}
A {\it local $\tilde\R$-Nash manifold} is a subset of an $\tilde\R$-Nash manifold $\mathcal M$ of the form $\phi ^{-1}(\mathfrak m_+)$ for some positive $\tilde\R$-Nash function $\phi$ on $\mathcal M$ (note that an $\tilde\R$-Nash manifold is a local $\tilde\R$-Nash manifold).
A {\it local $\tilde\R$-semialgebraic $\tilde\R$-continuous map} $\mathcal M_1\to \mathcal M_2$ between local $\tilde\R$-Nash manifolds is the restriction to $\mathcal M_1$ of an $\tilde\R$-semialgebraic $\tilde\R$-continuous map between the ambient $\tilde\R$-Nash manifolds.
\end{defi}

\begin{ex}\label{ex-f} Let $f$ be a polynomial function on $\R^n$. Then the set
$$\mathcal A_f=\{\gamma\in\tilde\R^n:\tilde f(\gamma)=t\}$$
is a non-singular $\tilde\R$-algebraic set and hence an $\tilde\R$-Nash manifold because the critical value set of $f$ is finite, and therefore there is no non-real critical value of $\tilde f$. Moreover
$$\{\gamma\in \mathcal A_f:\gamma(0)=0\}=\phi^{-1}(\mathfrak m_+),$$
where $\phi:\mathcal A_f\to\tilde\R$ is defined by $\phi(\gamma)=|\gamma|$, is a local $\tilde\R$-Nash manifold.
Such a set is sometimes called $T$-convex or definable in a weakly o-minimal structure \cite{MMS}.
\end{ex}

\subsection{Topology on $\wR$}
The $\tilde\R$-topology on $\tilde\R$ has many good properties when we treat $\tilde\R$ as an abstract real closed field.
However, this is not the case in this paper and for instance, the topology on $\R$ induced from the $\tilde\R$-topology on $\tilde\R$ is discrete.
We need to introduce another topology on $\tilde \R$, called {\it the product topology}.

Describe an element $\gamma \in \tilde \R$ as
$$\gamma (t)=\sum_{i\geq p} a_i t^{i/q},~~~ a_i \in \R ,~ (p,q) \in \Z\times \N^*$$
and regard $\wR$ as a vector subspace of $\Pi_{r\in\mathbf Q}\R t^r$ by the correspondence 
$$\gamma \to (...,a_i t^{i/q},...) \in \cdots \times \R t^{i/q} \times \cdots$$
Give the product topology to $\Pi_{r\in\mathbf Q}\R t^r$ and the induced topology on $\tilde\R$.
For an $\tilde\R$-Nash manifold $\mathcal M$ included in $\tilde\R^n$, we give to $\mathcal M$ the induced topology.

Then there are better relations between the topologies on $\tilde\R$ and $\R$, but nevertheless there are not enough many continuous maps.
Indeed, an $\tilde\R$-Nash map is not necessary continuous in this topology.
For example, the map $\tilde\R\ni\gamma\to\gamma^2\in\tilde\R$ is not continuous. Actually, for $\gamma_l=l t^{1/l}+t^{1-1/l}$, with $l\in \N^*$, we have $\gamma_l\to0$ because when we fix a finite number of exponents, the corresponding coefficients are equal to zero for $l$ big enough, whereas $\gamma^2_l\to\infty$ as $l\to\infty$ since the coefficient of $t$ is equal to $2l$.

We need to introduce a weaker notion of continuity. Actually, we introduce two kinds of weaker continuity, which happen to be equivalent to each other in the semialgebraic context (cf. Proposition \ref{prop-eq}). Our main result, Theorem \ref{main}.(2), holds for continuity in this topology.

 For $p\in\Z$ and $q\in\N^*$, let $\tilde\R_{p,q}$ denote the subset of $\tilde\R$ consisting of all Puiseux series of the form $\sum_{i\ge p}a_i t^{i/q}$.

\begin{defi}
We call a map $\phi:\tilde\R^n\to\tilde\R^m$ {\it weakly continuous} if 
$$\phi_{|{\tilde\R_{p,q}^n}}:\tilde\R_{p,q}^n\to\tilde\R^m$$
is continuous with respect to the induced topology on $\tilde\R_{p,q}^n$ and the product topology on $\tilde\R^m$ for any $(p,q)\in \Z\times \N^*$.

We call $\phi$ {\it finitely continuous} if $\phi_{|{\mathcal D}}:\mathcal D\to\tilde\R^m$ is continuous for any finite dimensional $\R$-linear subspace $\mathcal D$ of $\tilde\R^n$.
\end{defi}

Note that this definition is somehow natural from a singularity theory point of view, thinking for example that a $C^{\infty}$-map converges if the map, together with sufficiently enough derivatives, converge.

\begin{ex} A $\tilde\R$-polynomial map is an example of a weakly and finitely continuous map.
\end{ex}

\begin{lemma}\label{lem-b} Let $\phi: \wR^n \longrightarrow \wR$ be a $\wR$-semialgebraic $\wR$-continuous map. Then for all $(p,q) \in \Z \times \N^*$, there exists $(r,s) \in \Z \times \N^*$ such that
$$\phi(\wR^n_{p,q}) \subset \wR_{r,s}.$$
\end{lemma}

\begin{proof} Note first that $\wR_{p,q}$ is bounded, so that the image of $\wR^n_{p,q}$ by the $\wR$-continuous function $\phi$ is bounded in $\wR$.
Moreover, there exists a polynomial $P\in \wR[X_1,\ldots,X_{n+1}]$ such that $P(\gamma,\phi(\gamma))=0$ for any $\gamma \in \wR^n$ since $\phi$ is $\wR$-semialgebraic. Let $\tilde \C$ denote the algebraic closure of $\wR$, that is the complex algebraic Puiseux series field. For a fixed $\gamma \in \wR^n$, the number of $\tilde \C$-solutions of $P(\gamma, X)=0$ is finite, and we can choose $s\in \N$ so that this number is less than equal to $s$ for any choice of $\gamma \in \wR^n$.
We are going to prove that
$$\phi(\wR^n_{p,q}) \subset \cup_{r\in \Z} \wR_{r,s}.$$
If not, there exist $k>s$, $p_0\in \Z$ and $\gamma \in \wR^n_{p,q}$ such that the equation $P(\gamma,X)=0$ admits a solution $\delta$ of the form
$$\delta \in \wR_{p_0,k} \setminus \cup_{i=1}^{k-1} \wR_{p_0,i}.$$
But in that case the number of $\tilde \C$-roots of $P(\gamma,X)$ is strictly greater than $s$.

Finally, we conclude that there exists $r\in \Z$ such that $\phi(\wR^n_{p,q}) \subset \wR_{r,s}$ using the boundedness of $\phi(\wR^n_{p,q})$.
\end{proof}

\begin{prop}\label{prop-eq} Let $\phi: \wR^n \longrightarrow \wR^m$ be a $\wR$-semialgebraic weakly continuous map. Then $\phi$ is linearly continuous if and only if $\phi$ is weakly continuous.
\end{prop}

\begin{proof} It is sufficient to prove the case $m=1$.

Assume $\phi$ is weakly continuous. Any finite set of points in $\wR^n$ is included in some $\wR_{p,q}^n$, and the $\R$-linear space generated by those vectors is also included in $\wR_{p,q}^n$, since $\wR_{p,q}^n$ is also a $\R$-linear space. As a consequence $\phi$ is finitely continuous.

Assume now that $\phi$ is finitely continuous, and take $(p,q)\in \Z \times \N^*$. Note that, as a $\wR$-continuous $\wR$-semialgebraic map on the bounded set $\wR^n_{p,q}$, the function $\phi$ is uniformly $\wR$-continuous, namely 
$$\forall k\in \N, \exists l\in \N, \forall \gamma \in \wR^n_{p,q},~~\phi(\gamma +[-t^l,t^l])\subset \phi(\gamma)+[-t^k,t^k].$$
Moreover, there exists $(r,s) \in \Z \times \N^*$ such that
$$\phi(\wR^n_{p,q}) \subset \wR_{r,s}$$
by Lemma \ref{lem-b}.
In particular, for any $j\geq r$, the $t^{j/s}$-coefficient of the image $\phi(\gamma)$ in $\wR$ of a series $\gamma \in \wR^n_{p,q}$ is decided by a finite number of coefficients of $\gamma$. 
More precisely, if we denote by $\sum_{i\geq p}a_it^{i/q}$, with $a_i\in \R^n$, the elements of $\wR^n_{p,q}$ and by $\sum_{i\geq r}b_it^{i/s}$, with $b_i\in \R$, the elements of $\wR_{r,s}$, we can consider $\phi$ as a map
$$(a_1,a_2,\ldots)\mapsto (b_1,b_2,\ldots).$$
Then, for any $j\geq r$ there exists $i_0\in \N$ such that the map $\phi$ induces a map
$$(a_1,a_2,\ldots, a_{i_0})\mapsto b_j$$
which is continuous by finite continuity of $\phi$. As a consequence, $\phi$ is continuous on $\wR^n_{p,q}$, and so $\phi$ is weakly continuous.
\end{proof}

\begin{lemma}\label{lem-wcc} Let $\phi: \wR^n \longrightarrow \wR$ be a $\wR$-semialgebraic map. If $\phi$ is weakly continuous, then $\phi$ is $\wR$-continuous.
\end{lemma}

\begin{proof} Using the curve selection lemma (Theorem 2.5.5 in \cite{BCR}), it is sufficient to treat the case $n=1$. The function $\phi$ is piecewise $\wR$-continuous as a $\wR$-semialgebraic function. Assume that $\phi(0)=0$ and that $\phi$ is not $\wR$-continuous at 0. Let assume that $\phi$ admits a finite limit $\beta\in \wR$ when approaching 0 from above (note that the proof is similar is the limit is infinite), namely
$$\forall \epsilon \in \wR_+^*, \exists \eta_{\epsilon} \in \wR_+^*, \forall \gamma \in (0,\eta_{\epsilon}],~~|\phi(\gamma)-\beta|<\epsilon.$$
Assume $\beta(t)=bt^{p/q}+\cdots$ with $b\in \R^*$, and choose $\epsilon=t^k$ with $k>p/q$. Then for $l\in \N$ big enough, the series $t^l$ belongs to $(0,\eta_{t^k}]$, so that $\phi(t^l)=bt^{p/q}+\cdots$. But the series $t^l$ also belongs to $\wR_{0,1}$, so by weak continuity of $\phi$ the $t^{p/q}$-coefficient of $\phi(t^l)$ should converge to 0 as $l$ goes to infinity.
\end{proof}

\begin{prop}\label{prop-compo} The composition of weakly continuous $\wR$-semialgebraic maps is weakly continuous.
\end{prop}

\begin{proof} By Lemma \ref{lem-wcc} the maps are $\wR$-continuous, so we can use Lemma \ref{lem-b} to conclude.
\end{proof}

\subsection{Weak continuity in dimension one}

In this section, we discuss some properties of weak continuity specific to the one dimensional case. The results exposed here will be useful for the study of the general case in next section.

Let us begin with an illustrative example. 

\begin{ex}\label{ex-simple}
Let $\phi:\wR \longrightarrow \wR$ denote the function defined by $\phi(\gamma)=\gamma^{p/q}$ with $p,q\in \N^*$. If $\phi$ is weakly continuous, then $p/q\in \N$. 

Actually, suppose on the contrary that $p/q\notin \N$. Denote $k=[p/q]$, where $[\cdot]$ stands for the floor function. Note that $\phi$ is of class $C^{\infty}$ on $\wR_+^*$, so that for $x\in \R_+^*$ there exists $\gamma \in [0,1]_{\wR}$ such that
$$\phi(x+t)=\phi(x)+\phi'(x)t+\cdots+\frac{\phi^{(k+1)}(x)t^{k+1}}{(k+1)!}+\frac{\phi^{(k+2)}(x+\gamma (t) t)t^{k+2}}{(k+2)!}$$
by the mean value theorem.
Due to the particular form of $\phi$, note that $\phi(x),\cdots,\phi^{(k)}(x)$ tends to 0 in $\R$ as $x$ tends to 0 in $\R$, whereas $\phi^{(k+1)}(x)$ tends to infinity. Note moreover that $\phi^{(k+2)}(x+\gamma (t) t)$ is continuous at $t=0$, so that the $t^{k+1}$-coefficient of $\phi^{(k+2)}(x+\gamma (t) t)t^{k+2}$ is zero, and thus the $t^{k+1}$-coefficient of the right hand side of the equality above tends to infinity as $x$ goes to $0$ in $\R$. This fact leads to a contradiction because on the left hand side, since $x+t\in \wR_{0,1}$ and $\phi$ is continuous for the product topology on $\wR_{0,1}$, the $t^{k+1}$-coefficient of $\phi(x+t)$ should converge in $\R$ as $x$ goes to $0$ in $\R$.
\end{ex}

\begin{lemma}\label{lem-Nw} Let $\phi: [0,1]_{\wR} \longrightarrow \wR$ be a $\wR$-Nash function. Then there exists $\epsilon \in (0,1]_{\wR}$ such that $\phi$ is weakly continuous on $[0,\eps]_{\wR}$.
\end{lemma}

\begin{proof} Let assume that there exist $\epsilon\in (0,1]_{\wR}$ and $\beta\in (0,\infty)_{\wR}$ such that 
$$\forall \gamma \in [0,\eps]_{\wR}, \forall k \in \N,~~|\phi^{(k)}(\gamma)|\leq \beta ^k.$$
Then we can prove that $\phi$ is weakly continuous on $[0,\eps]_{\wR}$. Actually, for $\gamma_1,\ldots,\gamma_l\in [0,\eps]_{\wR}$, we are going to prove that $\phi$ is continuous on the $\R$-linear space $\mathcal D$ generated by $\gamma_1,\ldots,\gamma_l$. Choose $\delta \in \mathcal D$. Since the mean value theorem holds for $\phi$, for any $k\in \N$ and any $\gamma \in \mathcal D$, there exists $\theta\in [0,1]_{\wR}$ such that
$$\phi(\delta+\gamma)-\phi(\delta)=\phi'(\delta)\gamma+\cdots+\frac{\phi^{(k)}(\delta)}{k!}\gamma^k+\frac{\phi^{(k+1)}(\delta+\theta \gamma)}{(k+1)!}\gamma^{k+1}.$$
Fix $r\in \mathbf Q$. We want to prove that the $t^r$-coefficient of $\phi(\delta+\gamma)-\phi(\delta)$ goes to zero in $\R$ as $\gamma$ goes to zero in $\mathcal D$. Let us write $\gamma=\sum_{j=1}^l b_j\gamma_j$, with $b_1,\ldots,b_l\in \R$, and $\gamma_j=\sum_{i\in \N}c_{j,i}t^{i/q}$, with $c_{j,i}\in \R$ and $q\in \N^*$ a common denominator for $\gamma_1,\ldots,\gamma_l$ (note that the index $i$ runs in $\N$ because $\gamma_1,\ldots,\gamma_l\in [0,\eps]_{\wR}$).

Shrinking $\epsilon$ if necessary, let assume that $\epsilon \beta <t$. Then, for $k\in \N$ big enough the $t^r$-coefficient of $\phi^{(k+1)}(\delta+\theta \gamma)\gamma^{k+1}$ is equal to zero. Therefore the $t^r$-coefficient of $\phi(\delta+\gamma)-\phi(\delta)$ is a polynomial in $b_1,\ldots,b_l$, which proves the linear continuity of $\phi$.

To achieve the proof, it remains to justify the existence of $\epsilon$ and $\beta$ such that the inequality upstairs is valid. There exists a non zero polynomial $P\in \wR[x,y]$ such that $P(\gamma,\phi(\gamma))=0$ since $\phi$ is $\wR$-semialgebraic. Set $P(x,y)=P_0(x)y^d+\cdots+P_d(x)$ and $P_0(x)=x^eQ_0(x)$, with $e\in \N$ and $Q_0\in \wR[x]$ verifying $Q_0(0)\neq 0$. For $s\in \N$, multiply the equality $P(\gamma,\phi(\gamma))=0$ by $\gamma^{sd-e}$ so that we obtain
$$Q_0(\gamma)\psi^d(\gamma)+\gamma^{s-e}P_1(\gamma)\psi^{d-1}(\gamma)+\cdots+\gamma^{sd-e}P_d(\gamma)=0,$$
where $\psi(\gamma)=\gamma^s\phi(\gamma)$. Fix $s>e$ and set $Q_i(\gamma)=\gamma^{si-e}P_i(\gamma)$ for $i\in \{1,\ldots,d\}$. 
Then the polynomial $Q(x,y)=Q_0(x)y^d+\cdots+Q_d(x)$ satisfies $Q(\gamma, \psi(\gamma))=0$ with $Q_0(0)\neq 0$ and $Q_i(0)=0$ for $i\in \{1,\ldots,d\}$. Note that it is sufficient to prove the result for $\psi$ instead of $\phi$.
By deriving $Q(\gamma, \psi(\gamma))$, we find
$$0=\psi'(dQ_0\psi^{d-1}+\cdots+Q_{d-1})+Q_0'\psi^d+\cdots+Q_d'.$$
If $\psi$ is not constant on a neighbourhood of zero, there exists $\alpha\in \wR_+^*$ such that 
$$|dQ_0\psi^{d-1}+\cdots+Q_{d-1}|\geq \alpha$$
on $[0,\epsilon]_{\wR}$ for $\epsilon\in \wR_+^*$ sufficiently small. Then
$$|\psi'|\leq \frac{1}{\alpha}|Q_0'\psi^d+\cdots+Q_d'|$$
on $[0,\epsilon]_{\wR}$, and define $\beta$ to be equal to the right hand side of the inequality. Repeating the derivation of $Q(\gamma, \psi(\gamma))$, we see that
$$|\psi''|\leq \frac{1}{\alpha}|\psi'\big((d(d-1)Q_0\psi^{d-2}+\cdots+2Q_{d-2})+(dQ_0'+\cdots+Q_1')\big)+Q_0''\psi^d+\cdots+Q_d''|$$
so that $|\psi''|\leq \beta^2$ on $[0,\epsilon]_{\wR}$, by enlarging $\beta$ if necessary. We conclude that we can find a common $\beta$ such that $\psi^{(k)}\leq \beta^k$ on $[0,\epsilon]_{\wR}$ for all $k\in \N$ because the $k$-derivatives of $Q_i$, for $i\in \{0,\ldots,d\}$, vanish for $k$ big enough since the $Q_i$ are polynomials.
\end{proof}

Now we are in position to prove that a weakly continuous $\wR$-semialgebraic function is not only $\wR$-continuous as in Lemma \ref{lem-wcc}, but moreover $\wR$-Nash.

\begin{prop}\label{lem-wcN} Let $\phi: \wR \longrightarrow \wR$ be a $\wR$-semialgebraic function. If $\phi$ is weakly continuous, then $\phi$ is $\wR$-Nash.
\end{prop}

\begin{proof} We are going to simplify the expression of $\phi$ in order to come down to the form studied in Example \ref{ex-simple}. The problem is local, so consider the question at $0$ and assume $\phi(0)=0$.

First of all, there exist $q\in \N^*$ and $\epsilon \in \wR_+^*$ so that the function $\phi^q$ is Nash on $[0,\epsilon]_{\wR}$ since $\phi$ is $\wR$-semialgebraic (it follows from Proposition 8.1.13 in \cite{BCR} for example). So, shrinking $\epsilon$ if necessary, there exist $p\in \N^*$, $\epsilon'\in \wR_+^*$ and a $\wR$-Nash diffeomorphism $\rho :[0,\epsilon]_{\wR} \longrightarrow [0,\epsilon']_{\wR}$ such that $\phi^q=\rho^p$.

Note that, by shrinking $\epsilon'$ if necessary, $\rho^{-1}$ is weakly continuous by Lemma \ref{lem-Nw}. Then it suffices to treat the case of the composed function $\phi \circ \rho^{-1}$, which is weakly continuous by Proposition \ref{prop-compo}, and defined on $[0,\epsilon']$ by the formula $\phi \circ \rho^{-1} (\gamma)=\gamma^{p/q}$.

Note that if there exists a real number in $(0,\epsilon']_{\wR}$, the explanation given in Example \ref{ex-simple} are suitable. Assume now that $(0,\epsilon']\cap \R=\emptyset$, or equivalently that $\epsilon'\in \mathfrak m_+$. Then the multiplication by $\epsilon'$ gives a $\wR$-Nash diffeomorphism $\xi: [0,1]_{\wR}\longrightarrow [0,\epsilon']_{\wR}$. then $\phi\circ \rho^{-1} \circ \xi (\gamma)=\epsilon'^{p/q}\gamma^{p/q}$ is weakly continuous again, so $p/q\in \N$ by Example \ref{ex-simple}.

To achieve the proof, we need to see that if $\phi$ is of class $\wR$-Nash on $[-\epsilon,0]_{\wR}$ and on $[0,\epsilon]_{\wR}$, then $\phi$ is of class $\wR$-Nash on $[-\epsilon_1,\epsilon_1]_{\wR}$ for some $\epsilon_1\in \wR_+^*$. As $\phi$ is $\wR$-Nash on $[-\epsilon,0]_{\wR}$, there exist $\epsilon_1\in \wR_+^*$ and a $\wR$-Nash function $\phi_-$ defined on $[-\epsilon,\epsilon_1]_{\wR}$ which coincides with $\phi$ on $[-\epsilon,0]_{\wR}$. Shrinking $\epsilon_1$ if necessary, there exists a $\wR$-Nash function $\phi_+$ defined on $[-\epsilon_1,\epsilon]_{\wR}$ which coincides with $\phi$ on $[0,\epsilon]_{\wR}$. Define a function $\psi$ on $[\epsilon_1,\epsilon_1]_{\wR}$ to be zero on $[-\epsilon_1,0]_{\wR}$ and to be $\phi_+-\phi_-$ on $[0,\epsilon_1]_{\wR}$. Then $\psi$ is $\wR$-semialgebraic and weakly continuous. If $\psi$ is not identically zero, we can assume (as before) that $\psi(\gamma)=\gamma^{k}$ on $[0,\tilde a]_{\wR}$ for some $k\in \N^*$ and $a\in \R_+^*$. Then
$$
\psi(x+t)=\left\{
\begin{array}{l}
0\qquad\qquad\qquad\qquad\qquad\,\text{for}\ x\in[-a,\,0)\\
x^k+k t x^{k-1}+\cdots+t^k\quad\text{for}\ x\in[0,\,a).
\end{array}
\right.
$$
In particular, the $t^k$-coefficient of $\psi(x+t)$ is equal to $1$ for $x\in [0,a)$ and to $0$ for $x\in [a,0)$, in contradiction with the linear continuity of $\psi$ in restriction to the $\R$-linear space generated by $1$ and $t$. As a consequence $\psi$ is constant equal to zero, and therefore $\phi$ is of class $\wR$-Nash.
\end{proof}

\subsection{Finite dimensional case}

For the definition of manifold with corners, we refer to \cite{}.

\begin{prop}\label{prop-Nwc} Let $g: M_1 \longrightarrow M_2$ be a Nash map between Nash manifolds possibly with corners. Then for any compact $\wR$-semialgebraic subset $\mathcal X$ of $\tilde M_1$, the restriction $\tilde g_{| \mathcal X}$ is weakly continuous.
\end{prop}

\begin{proof} We can assume that $M_2=\R$ without loss of generality. Assume $M_1\subset \R^m$. 

For simplicity, we can reduce the proof to the case where $M_1$ is a $n$-dimensional compact Nash manifold with corners of $\R^n$ as follows. First, it suffices to prove that $\tilde g$ is continuous in restriction to $\tilde U$ for any compact semialgebraic neighbourhood $U \subset M_1$. In particular we may assume that $M_1$ is the graph of a Nash map $h$ from a $n$-dimensional compact Nash manifold with corners $M_3\subset \R^n$ to $\R^{m-n}$. Since the projection map from $\tilde M_1$ to $\tilde M_3$ is clearly weakly continuous, it is sufficient (using Proposition \ref{prop-compo}) to prove that the map $\tilde g \circ (\id \times h)$ from $\tilde M_3$ to $\tilde M_2$ is weakly continuous.

So we assume that $g$ is a Nash map from a $n$-dimensional compact Nash manifold with corners $M_1\subset \R^n$ to $\R$. We are going to prove that $\tilde g$ is finitely continuous, and conclude using Proposition \ref{prop-eq}. Note that $\tilde M_1\subset \cup_{q\in \N^*}\wR^n_{0,q}$ since $M_1$ is compact, so that an element $\gamma \in \tilde M_1$ has a well-defined endpoint $\gamma(0)\in M_1$. 

Let $\mathcal D$ be a $\R$-linear subspace of $\cup_{q\in \N^*}\wR^n_{0,q}$ generated over $\R$ by $\gamma_1,\ldots,\gamma_l$, and chose $\delta \in \mathcal D$. Set $\delta(0)=x_0\in M_1$. Since $g$ is Nash over the real numbers, describe $g$ (\cite{BCR} Proposition 8.1.8) around $x_0$ as a series
$$g(x)=\sum_{I\in \N^n} a_I(x-x_0)^I,~~a_I\in \R ,$$
and similarly, for $j\in \{1,\ldots,l\}$, denote
$$\gamma_j=\sum_{i\in \N}c_{j,i}t^{i/q},~~c_{j,i}\in \R^n$$
where $q\in \N^*$ is a common denominator for $\gamma_1,\ldots,\gamma_l$. Then, for $b_1,\ldots,b_l\in \R$, we have
$$\tilde g (\sum_{j=1}^lb_j\gamma_j)(t)=g(\sum_{j=1}^lb_j\gamma_j(t))=\sum_{I\in \N^n} a_I(\sum_{j=1}^l\sum_{i\in \N} b_jc_{j,i}t^{i/q}-x_0)^I$$
with $t\in [0,r]$, for $r\in \R_+^*$ sufficiently small. As a consequence, for any $k\in \mathbf Q$, there exists $i_0\in \N$ such that the $t^k$-coefficient of $\tilde g (\sum_{j=1}^lb_j\gamma_j)(t)$ is equal to the $t^k$-coefficient of
$$\sum_{I\in \N^n} a_I(\sum_{j=1}^l\sum_{i=0}^{i_0} b_jc_{j,i}t^{i/q}-x_0)^I,$$
for $t$ sufficiently small and $b_1,\ldots,b_l$ such that $\sum_{j=1}^lb_j\gamma_j(0)$ is closed enough to $x_0$. The latter function is Nash, therefore the $t^k$-coefficient of $\tilde g (\sum_{j=1}^lb_j\gamma_j)(t)$ is continuous in $b_1,\ldots,b_l$ and $\tilde g$ is continuous in restriction to $\mathcal D$.
\end{proof}

\begin{rmk}\label{rmk} \begin{flushleft}
\end{flushleft}
\begin{enumerate}
\item If $M_1$ is not compact, the map $\tilde g$ is not necessarily weakly continuous. Consider for example the function $g:\R\longrightarrow \R$ defined by $g(x)=(1+x^2)^{-1}$. Then $\tilde g(ct^{-1})=t^2/c^2-t^4-c^4+\cdots$ for $c\in \R^*$, so that $\tilde g(ct^{-1})$ does not tend to $\tilde g(0)=1$ as $c$ goes to 0 in $\R$.
\item A $\wR$-Nash map defined on a closed and bounded $\wR$-Nash manifold is not necessarily weakly continuous. Consider for example the map $\phi$ defined by $\phi(\gamma)=\tilde g (\gamma/t)$ on $[0,1]_{\wR}$, where $g$ is defined upstairs.
\end{enumerate}
\end{rmk}

Next result is a weaker analogue of Lemma \ref{lem-wcN} in higher dimensions.

\begin{prop}\label{prop-C1} Let $\phi: \mathcal M_1 \longrightarrow \mathcal M_2$ be a local $\wR$-semialgebraic map between local $\wR$-Nash manifolds possibly with corners. If $\phi$ is weakly continuous, then it is of class $\wR$-$C^1$. 
\end{prop}

\begin{proof} As the problem is local, we assume that $\phi$ is a $\wR$-semialgebraic weakly continuous function on $\wR^n$. The partial derivatives of $\phi$ exist by Proposition \ref{lem-wcN}, therefore it suffices to prove that they are $\wR$-continuous. Consider the partial derivative $\frac{\partial \phi}{\partial\gamma_1}$ and let us prove that it is $\wR$-continuous at $0\in \wR^n$. Let assume that $\frac{\partial \phi}{\partial\gamma_1}$ is not $\wR$-continuous at $0\in \wR^n$. In that case, by the curve selection lemma there exists a $\wR$-continuous curve $\eta:[0,1]_{\wR} \longrightarrow \wR^n$ satisfying $\eta(0)=0$, whose composition $\frac{\partial \phi}{\partial\gamma_1}\circ \eta$ is $\wR$-continuous on $(0,1]_{\wR}$ but the limit $\theta \in \wR^n$ of $\frac{\partial \phi}{\partial\gamma_1}\circ \eta(\gamma)$ as $\gamma$ tends to $0$ in $[0,1]_{\wR}$ is not equal to $\frac{\partial \phi}{\partial\gamma_1}(0)$.

We are going to obtain a contradiction with the weak continuity of $\phi$ by particularising this limit to a relevant set of series. 
Note that there exist $q\in \N$ and $\epsilon\in \wR_+^*$ such that the curve $\xi$ defined by $\xi(\gamma)=\eta(\epsilon \gamma^q)$ is a weakly continuous $\wR$-Nash map on $[0,1]_{\wR}$ by Proposition 8.1.13 in \cite{BCR} and Proposition \ref{lem-wcN}.

As $\theta$ is different from $\frac{\partial \phi}{\partial\gamma_1}(0)$, there exist $r\in \mathbf Q_+^*$ and $c\in \R_+^*$ such that
$$ct^r< |\theta -\frac{\partial \phi}{\partial\gamma_1}(0) |<2ct^r.$$
As a consequence, for $s\in \mathbf Q$ big enough, the inequalities
$$\frac{c}{2}t^r< |\frac{\partial \phi}{\partial\gamma_1}(\xi(at^s)) -\frac{\partial \phi}{\partial\gamma_1}(0) |<3ct^r\leqno(1)$$
hold for any $a\in \R_+^*$. We are going to replace the derivative of $\phi$ by a difference quotient in order to make use of the weak continuity of $\phi$. More precisely, for any $\gamma \in [0,1]_{\wR}$, there exist $\beta \in \wR_+^*$ such that
$$|\frac{\phi\big(\xi_1(\gamma)+\beta,\xi_2(\gamma),\ldots,\xi_n(\gamma)\big)}{\beta}-\frac{\partial \phi}{\partial\gamma_1}(\xi(\gamma))|<\frac{c}{8}t^r\leqno(2)$$
by the mean value theorem, and moreover the set of all such $(\gamma,\beta)$ is a $\wR$-semialgebraic subset of $[0,1]_{\wR}\times \wR_+^*$. In particular there exists a strictly increasing $\wR$-semialgebraic $\wR$-continuous function $\zeta$ on $[0,1]_{\wR}$ such that $\zeta^{-1}(0)=0$ and $(2)$ holds for any $(\gamma, \beta)\in (0,1]_{\wR}\times \wR_+^*$ satisfying $\beta \leq \zeta(\gamma)$. Restricting $\zeta$ to $\{at^s:a\in \R_+^*\}$, we see that there exists $\beta_0\in \wR_+^*$ such that $\beta_0<\zeta(at^s)$  for any $a\in \R_+^*$. As a consequence $(2)$ holds for $\beta=\beta_0$ and any $\gamma \in \{at^s:a\in \R_+\}$. Combining $(1)$ and $(2)$ provides 
$$\frac{c}{4}t^r<|\frac{\phi\big(\xi_1(at^s)+\beta_0,\xi_2(at^s),\ldots,\xi_n(at^s)\big)}{\beta_0} -\frac{\phi(\beta_0,0,\ldots,0)-\phi(0)}{\beta_0} | < 4ct^r \leqno(3)$$
for any $a\in \R_+^*$. But $(3)$ is in contradiction with the fact that, by weak continuity of $\phi$ and $\xi$, the $t^r$-coefficient of
$$\frac{\phi\big(\xi_1(at^s)+\beta_0,\xi_2(at^s),\ldots,\xi_n(at^s)\big)}{\beta_0}$$
converges to the $t^r$-coefficient of 
$$\frac{\phi(\beta_0,0,\ldots,0)-\phi(0)}{\beta_0}$$
as $a$ tends to $0$ in $\R$.
\end{proof}

\begin{rmk} In Proposition \ref{prop-C1}, we do not not whether the partial derivatives of $\phi$ are weakly continuous. this is the reason why we cannot prove that $\phi$ is of class $\wR$-Nash as in the one dimensional case.
\end{rmk}

We end this section with a technical result that will be useful in the proof of Lemma \ref{lem-homN2}.
Let $\Upsilon: \wR \longrightarrow \R$ denote the projection onto the constant term, namely if $\gamma=\sum_{i\geq p} a_i t^{i/q} \in \wR$, then $\Upsilon(\gamma)=a_0$.

\begin{lemma}\label{lem-tau} Let $\phi$ be an $\tilde\R$-semialgebraic weakly continuous function on $[0,\,1]_{\tilde\R}^n$ such that $\Ima \phi\subset[-\tilde a,\tilde a]_{\wR}$, with $a\in \R_+^*$.
Then 
$$\Upsilon\circ \phi=\Upsilon\circ \phi\circ(\Upsilon\times\cdots\times\Upsilon)$$
on $[0,\,1]_{\tilde\R}^n$, and restricting to the real numbers, we have that $\Upsilon\circ \phi|_{[0,\,1]^n}$ is a semialgebraic continuous function.
\end{lemma}

\begin{proof} We prove the equality announced by reduction to absurdity. Let $\gamma_0\in [0,\,1]_{\tilde\R}^n$ such that $\Upsilon \circ \phi (\gamma_0)\neq \Upsilon \circ \phi \circ(\Upsilon\times\cdots\times\Upsilon)(\gamma_0)$. We can assume than $n=1$ by restricting $\phi$ to the line over $\wR$ passing through $\gamma_0$ and $(\Upsilon\times\cdots\times\Upsilon)(\gamma_0)$, regarding that line as $\wR$ via the $\wR$-linear $\wR$-homeomorphism $\theta$ from $\wR$ to that line defined by $\gamma \mapsto \gamma \frac{\gamma_0}{|\gamma_0|}+(1-\gamma)(\Upsilon\times\cdots\times\Upsilon)(\gamma_0)$. Note that $\Upsilon \circ \theta^{-1}(\gamma_0)=0$.

As a consequence, we suppose $n=1$ and $\Upsilon( \gamma_0)=0$, and without loss of generality let assume moreover $\phi(0)=0$, so that $\Upsilon \circ \phi(\gamma_0)\neq 0$. Suppose $\Upsilon \circ \phi(\gamma_0)< 0$ for example. The function $\phi$ is monotone by pieces since $\phi$ is a $\wR$-semialgebraic function, so we may suppose that $\phi$ is increasing on $[\gamma_0,1]_{\wR}\cap \mathfrak m_+$ (or decreasing on a small neighborhoud on the left of $\gamma_0$, modifying $\gamma_0$ if necessary). Note that $\Upsilon\circ \phi$ is also increasing on $[\gamma_0,1]_{\wR}\cap \mathfrak m_+$.

Let $\gamma_1\in (\gamma_0,1]_{\wR}\cap \mathfrak m_+$ be such that $\gamma_0<x\gamma_1$ for any $x\in \R_+^*$. Denote by $\mathcal L_1$ the $\R$-line passing through 0 and $\gamma_1$, and by $\mathcal L_1^+$ the positive half line $\mathcal L_1^+=\{x\gamma_1,~~x\in \R^*_+\}$. Then the function $\Upsilon\circ \phi_{|\mathcal L_1}$ is continuous by weak continuity of $\phi$, it is increasing on $\mathcal L_1^+$ and $\Upsilon \circ \phi_{|\mathcal L_1}(0)=0$. Therefore $\Upsilon \circ \phi (\gamma)\geq 0$ for any $\gamma \in \mathcal L_1^+$.

Denote by $\mathcal L_2$ the $\R$-line passing through $\gamma_0$ and $\gamma_1$ and by $[\gamma_0,\gamma_1]_{\R}$ the segment between $\gamma_0$ and $\gamma_1$ in $\mathcal L_2$. For any $\gamma \in (\gamma_0,\gamma_1]_{\R}$, note that we have $\gamma_0<x\gamma$ for any $x\in \R_+^*$, so that $\Upsilon\circ \phi (\gamma)\geq 0$ as before. However $\Upsilon \circ \phi$ should be continuous on  $[\gamma_0,\gamma_1]_{\R}$ by weak continuity of $\phi$, which is in contradiction with $\Upsilon\circ\phi(\gamma_0)<0$.

It remains to prove that $\Upsilon\circ \phi$ is a semialgebraic function on $[0,1]^n\subset \R^n$, since the continuity of $\Upsilon\circ\phi_{|[0,1]^n}$ follows from the weak continuity of $\phi$. As $\phi$ is $\wR$-semialgebraic, let $P\in \wR[x_1,\ldots,x_{n+1}]$ be a nonzero polynomial such that $P(\gamma,\phi(\gamma))=0$ for any $\gamma \in [0,\,1]_{\tilde\R}^n$. Multiplying $P$ by the relevant power of $t$, we may suppose that all the coefficients of $P$ are bounded, and moreover some of them are not in $\mathfrak m_+$. Denote by $Q\in \R[x_1,\ldots,x_{n+1}]$ the polynomial obtained from $P$ by replacing the coefficients of $P$ with their value under $\Upsilon$. Then $Q$ is nonzero and $Q(x,\Upsilon\circ \phi (x))=0$ for any $x\in [0,1]^n$, so that $\Upsilon\circ \phi$ is semialgebraic on $[0,1]^n$.
\end{proof}

\section{Comparison of homologies}

In this section we compare different homology theories on local Nash manifolds, the usual singular homology, the $\wR$-semialgebraic singular homology (cf. \cite{DK} for the introduction of semialgebraic homology) where we consider $\wR$-semialgebraic $\wR$-continuous chains, and another sort of singular homology where we consider $\wR$-semialgebraic weakly continuous chains. 

\subsection{Algebraic topology over Puiseux series}\label{sect-algtop}

Let $\triangle^n$ denote the $n$-simplex spanned by $0,(1,0,...,),...,(0,...,0,1)$ in $\R^n$. For a topological space $X$, let $S_n(X)$ be the set of singular $n$-simplexes from $\triangle^n$ to $X$. We denote by $H_*(X)$ the singular homology groups of $X$ (with coefficient in $\mathbf Z$).

In this section, we establish several isomorphisms between different homology groups. We begin with the elementary remark that the singular homology groups $H_*(X)$ of a topological space $X$ can be defined by $\tilde\R$-simplexes as well as by usual $\R$-simplexes. Actually, replacing $\triangle^n$ with its extension $\tilde\triangle^n$ to $\wR^n$, consider the set  $S'_n(X)$ of singular $n$-$\tilde\R$-simplexes from $\tilde\triangle^n$ to $X$ and denote the corresponding homology groups by $H'_*(X)$.

\begin{lemma}\label{lem-h} The singular homology groups $H_*(X)$ of a topological space $X$ are isomorphic to the homology groups $H_*'(X)$ defined using $\tilde\R$-simplexes.
\end{lemma}

\begin{proof}
For $u\in S_n(X)$, define $\alpha(u)\in S'_n(X)$ by $\alpha(u)(\gamma)=u(\gamma(0))$ for $\gamma\in\tilde\triangle^n$, so that we obtain a map $\alpha: S_n(X) \longrightarrow S_n'(X)$. Similarly, for $\sigma \in S'_n(X)$, define $h(\sigma)\in S_n(X)$ by $h(\sigma)(x)=\sigma(\tilde x)$ for $x\in\triangle^n$, and this gives a map $h:S_n'(X) \longrightarrow S_n(X)$. Then $\alpha$ and $h$ define an homotopy equivalence. Indeed, $h\circ \alpha$ is the identity map of $S_n(X)$. Moreover the $\wR$-simplex $\alpha \circ h(\sigma)$, for $\sigma \in S_n'(X)$, is given by $\alpha \circ h(\sigma)(\gamma)=\sigma(\tilde {\gamma(0)})$ for any  $\gamma \in \tilde \triangle^n$. As a consequence, the map
$$\begin{array}{l}
\tilde \triangle^n\times [0,1] \longrightarrow X\\
\quad (\gamma,s) \quad \mapsto ~ \sigma\big((1-s)\gamma+s\tilde{\gamma(0)} \big)
\end{array}$$
defines an homotopy between $\sigma \in S_n'(X)$ and $\alpha \circ h(\sigma)$.
\end{proof}

  As we have already noticed, there are not enough many continuous $\tilde\R$-semialgebraic maps from $\tilde\triangle^n$ to $X$.
Hence we are interested in $\tilde\R$-semialgebraic $\tilde\R$-continuous maps and $\tilde\R$-semialgebraic weakly continuous maps rather than simply continuous maps.
We define below two kinds of homology groups which take into account this phenomenon. We will prove some isomorphisms between the corresponding homologies, in the spirit of Lemma \ref{lem-h} but with more involved proofs.

\begin{defi}
Let $\mathcal X$ and $\mathcal Y$ be $\tilde\R$-semialgebraic sets.
An {\it $\tilde\R$-semialgebraic singular $n$-simplex} of $\mathcal X$ is an $\tilde\R$-semialgebraic $\tilde\R$-continuous map from $\tilde\triangle^n$ to $\mathcal X$. We denote by $\tilde S_n(\mathcal X)$ the set of $\tilde\R$-semialgebraic singular $n$-simplex of $\mathcal X$, by $\tilde H_*(\mathcal X)$ the associated homology groups, and we call them the {\it $\tilde\R$-semialgebraic singular homology groups}. 

An {\it $\tilde\R$-semialgebraic $\tilde\R$-homotopy} $\theta_{\lambda}:\mathcal X\to \mathcal Y$, with $\lambda\in[0,\,1]_{\tilde\R}$, means an $\tilde\R$-semialgebraic $\tilde\R$-continuous map $(\gamma,\lambda)\mapsto \theta_{\lambda}(\gamma)$ from $\mathcal X\times[0,\,1]_{\tilde\R}$ to $\mathcal Y$.
If the $\theta_{\lambda}$ are all embeddings for $\lambda\in[0,\,1]_{\tilde\R}$, the $\tilde\R$-semialgebraic $\tilde\R$-homotopy is called an {\it $\tilde\R$-semialgebraic $\tilde\R$-isotopy}.
\end{defi}

The $\tilde\R$-semialgebraic $\tilde\R$-continuous maps between $\tilde\R$-semialgebraic sets and the homology groups $\tilde H_*$ satisfy the Eilenberg-Steenrod axioms of homology groups.
In the case where $\mathcal X$ is a local $\tilde\R$-Nash manifold, we define $\tilde H_*(\mathcal X)$ similarly.

Note that the family of local $\tilde\R$-Nash manifolds possibly with corners and the family of $\tilde\R$-semialgebraic weakly continuous maps between them form a category as shown already.
  
\begin{defi}
Let $\mathcal M$ be a local $\tilde\R$-Nash manifold possibly with corners. An {\it $\tilde\R$-semialgebraic weak singular $n$-simplex} of $\mathcal M$ is a $\tilde\R$-semialgebraic weakly continuous map from $\tilde\triangle^n$ to $\mathcal M$. We denote by $\tilde H^w_*(\mathcal M)$ the associated homology groups, and we call them the $\tilde\R$-semialgebraic weak singular homology groups of $\mathcal M$.
\end{defi}

For $\mathcal N$ another local $\tilde\R$-Nash manifold possibly with corners, we define similarly a {\it local $\tilde\R$-semialgebraic weak homotopy} $\theta_{\lambda}:\mathcal M\to \mathcal N$, with $\lambda\in[0,\,1]_{\tilde\R}$. The $\tilde\R$-semialgebraic weakly continuous maps between local $\tilde\R$-Nash manifolds possibly with corners and the homology groups $\tilde H^w_{*}$ satisfy the Eilenberg-Steenrod axioms of homology groups.

An $\tilde\R$-semialgebraic weak singular $n$-simplex is in particular an $\tilde\R$-semialgebraic singular $n$-simplex by Lemma \ref{lem-wcc}. As a consequence we have a natural map from $\tilde H^w_{*}(\mathcal M)$ to $\tilde H_*(\mathcal M)$, where $\mathcal M$ is a local $\tilde\R$-Nash manifold possibly with corners. This defines a functorial morphism between the covariant functors $\mathcal M\to\tilde H^w_{*}(\mathcal M)$ and $\mathcal M\to\tilde H_*(\mathcal M)$.

The goal of this section is to prepare the material to prove Theorem \ref{thm-hom} which states that the natural maps from $\tilde H^w_{*}(\mathcal F_f)$ to $\tilde H_*(\mathcal F_f)$ and from $\tilde H_*(\mathcal F_f)$ to $H_*(F_f(a,r))$ are isomorphisms, for $a\in \R_+^*$ and $r\in \R_+^*$ small enough (where $\mathcal F_f$ is a local $\wR$-Nash manifold and $F_f(a,r)$ is a semialgebraic set, cf part \ref{sect-homol}). In order to do this, we begin with considering the compact case in Lemma \ref{lem-comp}, proving that 
$\tilde H_*(\tilde X)$ is isomorphic to $H_*(X)$ for a compact semialgebraic set $X$.

\begin{lemma}\label{lem-comp}\begin{flushleft}
\end{flushleft}
\begin{enumerate}
\item A closed and bounded $\tilde\R$-semialgebraic set is $\tilde\R$-semialgebraically $\tilde\R$-homeomorphic to the $\tilde\R$-extension of some compact semialgebraic set $X$.
\item For such a set $\tilde X$ there is a natural isomorphism $\tilde H_*(\tilde X)\to H_*(X)$.
\end{enumerate}
\end{lemma}

The proof of $(2)$ follows from usual arguments of algebraic topology combined with the simplicial homotopy theorem (Lemma 3.1 in \cite{SI}). However, we write down the proof in full details because we will use similar arguments under more involved situations latter.

Let $X$ be a compact semialgebraic set, and let $K$ be a simplicial decomposition of $X$. Denote $\tilde K$ the extension of the simplexes of $K$, namely $\tilde K=\{\tilde \sigma: \sigma \in K\}$. Then $\tilde K$ is an $\tilde R$-simplicial complex whose (simplicial) homology $H_*(\tilde K)$ is isomorphic to the simplicial homology $H_*(K)$ of $K$ (and which is also isomorphic to the singular homology groups of $X$). We are going to relate $H_*(K)$ to $\tilde H_*(\tilde X)$ passing through homology groups related to $X$ and $K$.

Let denote by $\tilde S^{L}_n(\tilde X)$ the set of $\tilde R$-linear (relatively to $\tilde K$) maps from $\tilde \triangle^n$ to $\tilde X$, and by $\tilde H^{L}_n(\tilde X)$ the associated homology groups. Denote similarly by $\tilde S^{PL}_n(\tilde X)$ the set of $\tilde R$-piecewise linear maps from $\tilde \triangle^n$ to $\tilde X$, and by $\tilde H^{PL}_n(\tilde X)$ the associated homology groups. Then we have natural maps
$$H_n(X) \longrightarrow H_n(K) \longrightarrow H_n(\tilde K) \longrightarrow \tilde H_n^{L}(\tilde X) \longrightarrow \tilde H_n^{PL}(\tilde X) \longrightarrow \tilde H_n(\tilde X),$$
and the first fourth ones are isomorphisms by usual arguments in algebraic topology. The goal of the proof of the second part of Lemma \ref{lem-comp} is to see that the fifth one is also an isomorphism. 

Before entering into the details of the proof, we begin by recalling the statement of the simplicial homotopy theorem for the convenience of the reader.

\begin{lemma}(Lemma 3.1 in \cite{SI})\label{lem-homo} Let $\mathcal X$ and $\mathcal Y$ be closed and bounded $\wR$-polyhedra, and let $\phi:\mathcal X \longrightarrow \mathcal Y$ be a $\wR$-semialgebraic $\wR$-continuous map which is $\wR$-piecewise linear in restriction to a closed and bounded $\wR$-polyhedron $\mathcal X_0\subset \mathcal X$. Then $\phi$ is $\wR$-homotopic to a $\wR$-piecewise linear map, and the homotopy can be choosen to be fixed on $\mathcal X_0$.
\end{lemma}

\begin{proof}[Proof of Lemma \ref{lem-comp}] For the proof of $(1)$, note that for any real closed field $R$, a closed and bounded $R$-semialgebraic set is $R$-homeomorphic to a closed and bounded $R$-polyhedron (\cite{PL}, Theorem 2.2). Moreover, a closed and bounded $R$-polyhedron is the underlying polyhedron of some finite $R$-simplicial complex (\cite{PL}, Theorem 2.11). Finally, a finite $R$-simplicial complex is defined by a finite set and finite relations between the elements of the set. As a consequence, a closed and bounded $\wR$-semialgebraic set is $\wR$-semialgebraic $\wR$-homeomorphic to $\tilde X$ for a compact polyhedron $X$.

\vskip 5mm

For the proof of $(2)$, consider a compact semialgebraic set $X$ and let $K$ be a simplicial decomposition of $X$. We are going to prove that the natural map 
$$\tilde H_n^{PL}(\tilde X) \longrightarrow \tilde H_n(\tilde X)$$
is an isomorphism.

Let focus first on injectivity. Let $\xi$ be a $\wR$-piecewise linear $n$-chain on $\tilde X$, and assume that $\xi$ is the boundary of a $\wR$-semialgebraic singular $n+1$-chains, namely there exist $l\in \mathbf N$, $\sigma_l \in \tilde S_{n+1}(\tilde X)$ and $m_i\in \mathbf Z$ for $i\in \{1,\ldots,l\}$, such that 
$$\xi=\sum_{i=1}^lm_i\partial \sigma_i$$
as a $\wR$-semialgebraic singular $n$-chain. We are going to deform $\sigma_1,\ldots,\sigma_l$ into $\wR$-piecewise linear $n+1$-simplexes using Lemma \ref{lem-homo}. Let $\Sigma$ denote the disjoint union of $l$ copies of $\tilde \triangle^{n+1}$, and define 
$$\sigma: \Sigma\longrightarrow \tilde X$$ 
to be the $\wR$-semialgebraic $\wR$-continuous map whose restriction to the $i$-th copy of $\tilde \triangle^{n+1}$ coincides with $\sigma_i$, for $i\in \{1,\ldots,l\}$. Let $\Sigma_0\subset \Sigma$ denote the union of those faces of the $i$-th copy of $\tilde \triangle^{n+1}$ where the $\wR$-semialgebraic map $\partial \sigma_i$ is already $\wR$-piecewise linear, for any $i\in \{1,\ldots,l\}$. Identify two $n$-dimensional faces $\tilde \triangle_{i_1}$ and $\tilde \triangle_{i_2}$ of the copies of $\tilde \triangle^{n+1}$ in $\Sigma$ through an $\wR$-linear isomorphism $\psi:\tilde \triangle_{i_1} \longrightarrow \tilde \triangle_{i_2}$ as soon as $\partial \sigma_{i_1}=\partial \sigma_{i_2} \circ \psi$ on $\tilde \triangle_{i_1}$. We denote by $\Sigma '$ the resulting closed and bounded $\wR$-polyhedron and by $\pi: \Sigma \longrightarrow \Sigma'$ the associated projection. 
Define a $\wR$-semialgebraic $\wR$-continuous map $\sigma': \Sigma'\longrightarrow \tilde X$ by $\sigma '\circ \pi=\sigma$, and set $\Sigma_0'=\pi(\Sigma_0)$. By Lemma \ref{lem-homo}, there exists a $\wR$-semialgebraic $\wR$-continuous map
$\theta:\Sigma'\times [0,1]_{\wR} \longrightarrow \tilde X$ such that 
\begin{enumerate}
\item[-] $\theta_{|\Sigma'\times\{0\}}$ coincides with $\sigma'$, 
\item[-] the restriction to $\Sigma_0'$ remains fixed, namely $\theta_{|\Sigma_0'\times [0,1]_{\wR}}=\sigma'_{|\Sigma_0'} \times \id$, 
\item[-] and finally $\theta_{|\Sigma'\times\{1\}}$ is $\wR$-piecewise linear. 
\end{enumerate}
As a consequence, denoting by $\sigma''_i:  \tilde \triangle^{n+1} \longrightarrow \tilde X$ the restriction of $\theta_{|\Sigma'\times\{1\}}$ to the $i$-th copy of $\tilde \triangle^{n+1}$, we obtain $\wR$-piecewise linear $n+1$-simplexes $\sigma_i''\in \tilde S_{n+1}^{PL}(\tilde X)$, for $i\in \{1,\ldots,l\}$. Moreover the equality
$$\sum_{i=1}^lm_i\partial \sigma_i=\sum_{i=1}^lm_i\partial \sigma_i''$$
holds by construction of $\Sigma'$, so that $\xi$ is the boundary of a $\wR$-piecewise linear chain.

\vskip 5mm

Concerning the surjectivity, consider a $\wR$-semialgebraic $n$-cycle $\xi$ on $\tilde X$, namely $\xi=\sum_{i=1}^lm_i \sigma_i$ with $\sigma_i \in \tilde S_n(\tilde X)$ and $\partial \xi=0$. We look for a $\wR$-piecewise linear $n$-chain $\mu$ on $\tilde X$ and a $\wR$-semialgebraic $n+1$-chain $\nu$ on $\tilde X$ such that $\xi=\mu-\partial \nu$. Define the polyhedrons $\Sigma$ and $\Sigma'$, the projection $\pi:\Sigma \longrightarrow
 \Sigma'$ together with the maps $\sigma: \Sigma \longrightarrow \tilde X$ and $\sigma': \Sigma' \longrightarrow \tilde X$ similarly as above in the proof of injectivity. 
 
 

We use Lemma \ref{lem-homo} which gives the existence of an $\wR$-semialgebraic $\wR$-continuous map $\theta':\Sigma'\times [0,1]_{\wR} \longrightarrow \tilde X$ such that $\theta'_{|\sigma'\times \{0\}}$ is equal to $\sigma'$ and $\theta'_{|\sigma'\times \{1\}}$ is $\wR$-piecewise linear. Set $\theta=\theta'\circ (\pi\times \id)$. Then
$$\theta:\Sigma\times [0,1]_{\wR} \longrightarrow \tilde X$$
is a $\wR$-semialgebraic $\wR$-continuous map such that $\theta_{|\sigma\times \{0\}}$ is equal to $\sigma$ and $\theta_{|\sigma\times \{1\}}$ is $\wR$-piecewise linear. Restricting to the $i$-th copy of $\tilde \triangle^n$ in $\Sigma$, we obtain moreover $\wR$-semialgebraic $\wR$-continuous maps
$$\theta_i:\tilde \triangle^n\times [0,1]_{\wR} \longrightarrow \tilde X$$
for $i\in\{1,\ldots,l\}$. These maps satisfy 
$$\sum_{i=1}^l m_i \partial ({\theta_i}_{|\tilde \triangle^n\times \{ \lambda \}})\leqno(1)$$
for any $\lambda \in [0,1]_{\wR}$. We set
$$\mu=\sum_{i=1}^l m_i {\theta_i}_{|\tilde \triangle^n\times \{ 1\}}.$$
Then $\mu$ is a $\wR$-piecewise linear $n$-chain on $\tilde X$ and $\partial \mu=0$. Note moreover that 
$$\xi-\mu=\sum_{i=1}^l m_i ({\theta_i}_{|\tilde \triangle^n\times \{ 0\}}-{\theta_i}_{|\tilde \triangle^n\times \{ 1\}}).\leqno(2)$$
In order to define $\nu$, consider the canonical $\tilde\R$-simplicial decomposition $\mathcal L$ of $\tilde\triangle^n\times[0,\,1]_{\tilde\R}$, namely each $\tilde\R$-simplex in $\mathcal L$ of dimension $n+1$ is spanned by 
$$(\tilde v_1,0),...,(\tilde v_k,0),(\tilde v_k,1),\ldots,(\tilde v_{n+1},1)$$ 
(and the vertices have this order) where 
$$v_1=0,\,v_2=(1,0,...,0),...,v_{n+1}=(0,...,0,1).$$
We denote by $\mathcal L_i$ the $\tilde\R$-simplicial decomposition of the $i$-th copy of $\tilde\triangle^n\times[0,\,1]_{\tilde\R}$ in $\Sigma\times[0,\,1]_{\tilde\R}$. Then the induced $\tilde\R$-simplicial decomposition $\mathcal L_\Sigma=\cup_{i=1}^l \mathcal L_i$ of $\Sigma\times[0,\,1]_{\tilde\R}$, satisfies
$$\mathcal L_\Sigma|_{\Sigma\times\{0\}}=\mathcal K\times\{0\},\ \mathcal L_\Sigma|_{\Sigma\times\{1\}}=\mathcal K\times\{1\},\ \mathcal L_\Sigma^0=\mathcal K^0\times\{0,1\}.$$
Note that the vertices of each $\tilde\R$-simplex in $\mathcal L_\Sigma$ have an order.
For each $i\in \{1,\ldots,l\}$ and $\omega\in \mathcal L_i$ of dimension $n+1$, consider $\omega$ as the standart $n+1$-simplexe $\tilde\triangle_{n+1}$ through the unique $\tilde\R$-linear isomorphism from $\tilde\triangle_{n+1}$ to $\omega$ which preserves the orders of the vertices, and set $\theta_{i,\omega}=\theta_i|_\omega$.
Set
$$\nu=\sum_{i=1}^l\sum_{\omega\in \mathcal L,\dim\omega=n+1}m_i \theta_{i,\omega}.$$
Then the boundary of $\nu$ is equal to $\xi-\mu$ by $(1)$ and $(2)$.
\end{proof}

We are going to use the method of proof of Lemma \ref{lem-comp} in order to prove that, if $\mathcal M$ is a closed and bounded $\tilde\R$-Nash manifold possibly with corners and induced from a Nash manifold $N$, then $\tilde H^w_{*}(\mathcal M)\to H_*(N)$ is an isomorphism.

\begin{prop}\label{prop-comp} Let $M$ be a compact Nash manifold possibly with corners.
Then there is a natural isomorphism $\tilde H^w_{*}(\tilde M)\to H_*(M)$.
\end{prop}

In order to prove Proposition \ref{prop-comp}, we introduce as a tool another homology groups defined using Nash simplexes. Let $M$ be a Nash manifold possibly with corners.
Denote by $S^N_{n}(M)$ the set of Nash $n$-simplexes, namely Nash maps from $\triangle^n$ to $M$, and denote by $H^N_{*}(M)$ the corresponding homology groups. Similarly, denote by $\tilde S^N_{n}(\tilde M)$ the set of $\tilde\R$-extensions $\tilde u: \tilde \triangle^n \longrightarrow \tilde M$ of Nash $n$-simplexes $u\in S^N_{n}(M)$, and define the corresponding homology groups $\tilde H^N_{*}(\tilde M)$.

\begin{proof}[Proof of Proposition \ref{prop-comp}]
We have natural maps
$$\tilde H_*^w(\tilde M) \longleftarrow \tilde H^N_{*}(\tilde M) \longleftarrow H^N_{*}(M)\longrightarrow H_*(M),$$
the middle one being clearly an isomorphism. We prove in Lemma \ref{lem-homN} and Lemma \ref{lem-homN2} below that the two others are also isomorphisms.
\end{proof}

\begin{lemma}\label{lem-homN} Let $M$ be a compact Nash manifold possibly with corners.
Then the natural map from $H^N_{*}(M)$ to $H_*(M)$ is an isomorphism. 
\end{lemma}

\begin{proof} The proof follows the same lines as the proof of Lemma \ref{lem-comp}. The only difference is that we need a counterpart for Lemma 2.1 in \cite{SI}.

If $M$ has corners, let $N$ be a compact Nash manifold with corners which contains $M$ in its interior and such that there exists a Nash isotopy $H_s:M \longrightarrow N$, with $s\in [0,1]$, such that $H_0=\id$ and $\Ima H_1 =N$ (such an isotopy exists by the proof of Theorem VI.2.1 in \cite{Shiota}). In case $M$ has no corners, set $N=M$.
Let $\Sigma$ be the union of $l\in \mathbf N$ copies of the standard $n$-simplex $\triangle^n$, and consider as in the proof of Lemma \ref{lem-comp} a union $\Sigma_0$ of some faces of the copies of $\triangle^n$ in $\Sigma$, together with a quotient space $\Sigma'$ obtained by identifying some proper faces of the copies of $\triangle^n$ through via linear isomorphisms, with the quotient map $\pi:\Sigma \longrightarrow \Sigma'$.

Let $g':\Sigma' \longrightarrow M$ be a continuous semialgebraic map such that the restriction of the continuous semialgebraic map $g=g'\circ \pi$ to any face in $\Sigma_0$ is a Nash map. Then what we are going to prove is the following statement.

\textit{(*) There exists an homotopy $h_s': \Sigma' \longrightarrow M_1$, for $s\in [0,1]$, such that the induced homotopy $h_s=h_s'\circ \pi$ satisfies $h_0=g$, in restriction to $\Sigma_0$ the maps $h_s$ coincide with $g$ and $h_1$ is a Nash map.}

In order to prove this statement, consider first the case where $l=1$ so that $\Sigma=\triangle^n$, and assume moreover that $g_{|\partial \triangle^n}$ is of class Nash. Thus we can suppose $\Sigma_0=\partial \Sigma$ and $\Sigma'=\Sigma$. Assume $N$ is included in $\mathbf R^m$ and let $v:U\longrightarrow N$ be a Nash tubular neighbourhood of $N$ in $\mathbf R^m$. In order to solve the problem for such $g$, it suffices to find a strong Nash approximation $g_1:\triangle^n\longrightarrow \mathbf R^m$ of $g$ in the $C^0$-topology such that $g_1=g$ on $\partial \triangle^n$ and $\Ima g_1\subset U$ because in that case the maps defined by
$$h_s(x)=v\big((1-s)g(x)+sg_1(x) \big),~~s\in [0,1]$$
gives a relevant homotopy. To construct such an approximation, it sufficies to consider the case $m=1$. Extend $g_{|\partial \triangle^n}$ to a Nash function $g_2$ on $\triangle^n$ by Proposition 0.7 in \cite{CRS}, so that, replacing $g$ with $g-g_2$ if necessary, we can suppose that $g$ vanishes on the boundary of $\triangle^n$. Replacing again $g$ with $g\circ w$, where $w:\triangle^n\longrightarrow \triangle^n$ is a continuous semialgebraic map closed to the identity such that the inverse image $w^{-1}(\partial \triangle^n)$ of the boundary of $\triangle^n$ is a neighbourhood of $\partial \triangle^n$ in $\triangle^n$, we can even suppose that $g$ vanishes on a neighbourhood of the boundary of $\triangle^n$.
Now we construct the approximation as follows. Let $l_j$ be linear functions on $\R^n$ whose zero sets are the linear spaces spanned by the faces of $\triangle^n$ of dimension $n-1$.
Define a continuous semialgebraic function $q$ on $\triangle^n$ by
$$
q=\left\{
\begin{array}{l}
0\qquad\qquad\,\text{on}\ \partial\triangle^n\\
\frac{g}{\prod_j l_j}\quad\qquad\,\text{on}\Int \triangle^n.
\end{array}
\right.
$$
Let $p$ be a Nash approximation (for example a polynomial approximation) of $q$.
Then $p\prod_j l_j$ is a relevant Nash approximation of $g$.

The proof of the general statement \textit{(*)} follows from this particular case by induction on the number $l$ of copies of $\triangle^n$ in $\Sigma$ and on the dimension $n$. More precisely, if $l>1$, let $\Sigma_1$ be one copy of $\triangle^n$ in $\Sigma$ and let $\Sigma_2$ be the union of the other copies. By the induction hypothesis we obtain a homotopy $h_s: \pi(\Sigma_2) \longrightarrow M_1$, and this homotopy induces maps from the union of faces in $\Sigma_1$ that are identified with faces in $\Sigma_2$ via $\pi$, namely maps from $\Sigma_1\cap \pi^{-1}(\pi(\Sigma_1)\cap \pi(\Sigma_2))$. We can extend this maps to give an homotopy on $\pi(\Sigma_1)$, fixed on $\pi(\Sigma_2\cap \Sigma_0)$. Using the induction hypothesis for the case $l=1$, and extending if necessary $N$ to a bigger Nash manifold with corners, we can ask moreover that the homotopy ends with a Nash map. Therefore it suffices to treat the case $l=1$.
Furthemore, we can assume that the restriction of $g$ to $\partial \Sigma$ is Nash by the induction hypothesis on the dimension, so that the particular case treated upstairs enables to achieve the proof.
\end{proof}

\begin{lemma}\label{lem-homN2} Let $M$ be a compact Nash manifold possibly with corners.
The natural map $\tilde H^N_{*}(\tilde M)\to \tilde H^w_{*}(\tilde M)$ is an isomorphism.
\end{lemma}

\begin{proof} 
We proceed as in the proof of Lemma \ref{lem-homN} (note that the orthogonal projection of a tubular $\wR$-neighbourhood of $\tilde M$ in $\wR^m$ is weakly continuous because it is induced from a tubular neighbourhood of $M$ in its ambient Euclidean space $\R^m$). 
In particular, it suffices find a substitute for the approximation argument. More precisely, we are going to prove that for any $a\in\R_+^*$, an $\wR$-semialgebraic weakly continuous function $\phi$ on $\tilde \triangle^n$, whose restriction to $\partial \tilde \triangle^n$ is the $\wR$-extension of some Nash function $h$ on $\partial\triangle^n$, is $\tilde a$-approximated by the $\wR$-extension $\tilde g$ of some Nash function $g$ on $\triangle^n$ whose restriction to $\partial \triangle^n$ coincides with $h$.

Indeed, the function 
$$\Upsilon \circ \phi_{| \triangle^n}: \triangle^n \longrightarrow \R$$
is a continuous semialgebraic function by Lemma \ref{lem-tau}, and its restriction to $\partial \triangle^n$ coincides with the Nash function $h$. Let $g$ be a Nash $a/2$-approximation of $\Upsilon \circ \phi_{| \triangle^n}$ whose restriction to $\partial \triangle^n$ coincides with $h$. Then $\tilde g$ gives a relevant approximation of $\phi$.
\end{proof}


\section{Homology of the Milnor fibre}\label{sect-homol}

Let $f$ be a polynomial function on $\R^n$. We may associate to $f$ positive and negative Milnor fibres as in \cite{MCP}. The positive Milnor fibre $F_f(r,a)$ of $f$ at $x_0\in f^{-1}(0)\subset \R^n$ is the semialgebraic subset of $\R^n$ defined by
$$F_f(r,a)=\{x\in \R^n:~~ |x-x_0| \leq r,~~f(x)=a\}$$
where $a>0\in \R$ and $r>0\in \R$ are sufficiently small so that $f$ is a locally trivial fibration on a neighbourhood of $x_0$ over $(0,a]$ with fibre $F_f(r,a)$.
We associate to $f$ and $x_0$ a set of Puiseux series $\mathcal F_f$ by
$$\mathcal F_{f,x_0}=\{\gamma \in \wR^n:~~ \gamma (0)=x_0,~~ f\circ \gamma (t) =t\}.$$
Note that $\mathcal F_{f,x_0}$ is a local $\wR$-Nash manifold, so that we can discuss about the $\wR$-semialgebraic singular homology groups and the $\wR$-semialgebraic weak singular homology groups of $\mathcal F_{f,x_0}$.
In this part, we are going to compare these homology groups of $\mathcal F_{f,x_0}$ with the (classical) homology groups of $F_f(r,a)$. To this aim, we study the Nash triviality of the family of Milnor fibres $F_f(r,a)$, with $a>0\in \R$ and $r>0\in \R$, together with 
its analog after extension into the real field of Puiseux series, namely the semialgebraic subsets $\tilde F_{\tilde f}(\rho,\alpha)$ defined by
$$\tilde F_{\tilde f}(\rho,\alpha)=\{\gamma \in \wR^n:~~ |\gamma-\tilde x_0| \leq \rho,~~\tilde f(\gamma)=\alpha \}$$
with $\alpha>0 \in \wR$ and $\rho>0 \in \wR$.

\vskip 5mm

Note that $F_f(r,a)$ is a Nash manifold with boundary and that $\tilde F_{\tilde f}(\rho,\alpha)$ is an $\tilde\R$-Nash manifold with boundary. 

In the following, we fix $x_0=0\in \R^n$ and denote simply $\mathcal F_{f,x_0}=\mathcal F_f$ for simplicity of notation. Remark that $\mathcal F_f \subset\{\pm\gamma:\gamma\in \mathfrak m_+\}^n$. Moreover, for $\alpha=t$ the extension $\tilde F_{\tilde f}(\rho,t)$ is equal in that case to
$$\tilde F_{\tilde f}(\rho,t)=\{\gamma \in \wR^n:~~ |\gamma| \leq \rho,~~\tilde f(\gamma)=t \}.$$

We can recover $\mathcal F_f$ as a union of some of these extended Milnor fibres.

\begin{lemma}\label{decomp} 
For any $\rho_0 \in \mathfrak m_+\cup \{0\}$, the set $\mathcal F_f$ is equal to the union of the $\tilde F_{\tilde f}(\rho,t)$ over $\rho \in \mathfrak m_+$ with $\rho > \rho_0$, namely
$$\mathcal F_f=\cup_{\rho \in \mathfrak m_+,~\rho > \rho_0} \tilde F_{\tilde f}(\rho,t).$$
\end{lemma}

\begin{proof} Note first that $\tilde F_{\tilde f}(\rho,t) \subset \mathcal F_f$ for $\rho \in \mathfrak m_+$ since for $\gamma \in \wR$, if $|\gamma| \leq |\rho|$ then the limit of $\gamma (t)$ when $t$ goes to zero exist and satisfies $|\gamma(0)|\leq |\rho(0)|=0$. Conversely, it suffices to notice that for any $\gamma \in \mathfrak m_+$ is less than some element in $\mathfrak m_+$.
\end{proof}

\subsection{Nash triviality}

Let $M_1$ and $M_2$ be Nash manifolds possibly with corners, and let $M_3$ be a semialgebraic subset of $M_2$.
A Nash map $g:M_1\to M_2$ is called {\it Nash trivial} over $M_3$ if there is a Nash diffeomorphism $h:g^{-1}(x)\times M_3\to g^{-1}(M_3)$ for some $x\in M_3$ such that $g\circ h$ is the projection $g^{-1}(x)\times M_3\to M_3$.

We define similarly the $\tilde\R$-{\it Nash triviality} of a $\tilde\R$-Nash map.

\begin{rmk}
If $g:M_1\to M_2$ is Nash trivial over $M_3$, then $\tilde g:\tilde M_1\to\tilde M_2$ is $\tilde\R$-Nash trivial over $\tilde M_3$.
Indeed, let $h:g^{-1}(x)\times M_3\to g^{-1}(M_3)$ be a Nash diffeomorphism of Nash trivialization of $g: M_1\to M_2$.
Then $\tilde h:\tilde g^{-1}(\tilde x)\times\tilde M_3\to\tilde g^{-1}(\tilde M_3)$ is a $\tilde\R$-Nash $\tilde\R$-diffeomorphism of $\tilde\R$-Nash trivialization of $\tilde g:\tilde M_1\to\tilde M_2$.
\end{rmk}

We use classical triviality results in the Nash category to deal with the triviality of real Milnor fibres.

\begin{lemma}\label{lem-triv} \begin{flushleft}
\end{flushleft}
\begin{enumerate}
\item There exist $r_0\in\R_+^*$ and a non-negative continuous semialgebraic function $v$ on $[0,r_0]$ with zero set $\{0\}$ such that the projection map
$$\pi: \bigcup_{0<\alpha<\rho\in\tilde\R}\tilde F_{\tilde f}(\rho,\alpha)\times\{(\rho,\alpha)\} \longrightarrow \tilde\R^2$$
given by $(\gamma,\rho,\alpha)\mapsto(\rho,\alpha)$ is $\tilde\R$-Nash trivial over 
$$\mathcal D=\{(\rho,\alpha)\in\tilde\R^2:0<\rho\le\tilde r_0,\,0<\alpha\le\tilde v(\rho)\}.$$ 
In particular, $\tilde F_{\tilde f}(\rho_1,\alpha_1)$ and $\tilde F_{\tilde f}(\rho_2,\alpha_2)$ are $\tilde\R$-Nash $\tilde\R$-diffeomorphic for $(\rho_1,\alpha_1)$ and $(\rho_2,\alpha_2)$ in $\mathcal D$.

\item Moreover, the map $\pi$ is weakly continuously $\tilde\R$-Nash trivial over 
$$\mathcal D'=\{(\rho,\alpha)\in\tilde\R^2:\rho_0\le\rho\le\tilde r_0,\,\alpha_0\le\alpha \le\tilde v(\rho)\}$$
for any $0<\alpha_0<\rho_0\in\mathfrak m_+$.
In particular, $\tilde F_{\tilde f}(\rho_1,\alpha_1)$ and $\tilde F_{\tilde f}(\rho_2,\alpha_2)$ are weakly continuously $\tilde\R$-Nash $\tilde\R$-diffeomorphic for $(\rho_1,\alpha_1)$ and $(\rho_2,\alpha_2)$ in $\mathcal D'$.
\end{enumerate} 
\end{lemma}

\begin{proof} The proof of $(1)$ is a consequence of Theorem 3 in \cite{CS2}. To see this, note that the projection map
$$p: \bigcup_{0<a<r\in\R} F_{f}(r,a)\times\{(r,a)\} \longrightarrow \R^2$$
given by $(x,r,a)\mapsto(r,a)$ and its restriction to the boundaries
$$ \bigcup_{0<a<r\in\R} \partial F_{f}(r,a)\times\{(r,a)\} \longrightarrow \R^2$$
are proper and submersive onto
$$\{(r,a)\in \R^2:~0<a\ll r \ll 1\}.$$
In particular, there exist $r_0\in\R_+^*$ and a non-negative continuous semialgebraic function $v$ on $[0,r_0]$ with zero set $\{0\}$, such that these maps are proper and submersive onto
$$D=\{(r,a)\in \R^2:~r\leq r_0,~0<a\leq v(r)\}.$$
As a consequence $p$ is Nash trivial over $D$ by Theorem 3 in \cite{CS2}, and therefore $\pi$ is $\tilde R$-Nash trivial over $\tilde D=\mathcal D$. The second statement of the lemma follows from Proposition \ref{prop-Nwc}.
\end{proof}

Recall that $\mathcal A_f$ is the $\tilde\R$-algebraic set defined by $\tilde f$, namely
$$\mathcal A_f=\{\gamma\in\tilde\R^n:\tilde f(\gamma)=t\}.$$

\begin{cor}\label{cor-wctriv} The map 
$$\bigcup_{\rho\in\wR}\{\gamma\in \mathcal F_f:|\gamma|\le\rho\}\times\{\rho\}\longrightarrow \wR$$
given by $(\gamma,\rho)\mapsto \rho$
 and hence the map $\gamma\mapsto|\gamma|$ from $\mathcal F_f$ to $\wR$ are weakly continuously trivial over $\{\rho\in\mathfrak m_+:\rho\ge\rho_0\}$ for some $\rho_0\in\mathfrak m_+$.

Moreover, the map
$$\bigcup_{\rho\in\wR}\{\gamma\in\mathcal A_f,\,|\gamma|\le\rho\}\times\{\rho\}\longrightarrow \tilde\R$$
given by $(\gamma,\rho)\mapsto \rho$ and the map $\gamma\mapsto|\gamma|$ from $\mathcal A_f$ to $\wR$ are weakly continuously trivial over $[\rho_0,\,\tilde r_0]_{\tilde\R}$ for some $r_0\in\R$.
\end{cor}

\begin{proof} We apply Lemma \ref{lem-triv} by specifying $\alpha$ to the value $t$. This is possible because $t< \tilde v (\tilde r_0)$ since the semialgebraic function $v$ provided by Lemma \ref{lem-triv} is defined over $\R$, so that $\tilde v (\tilde r_0)$ belongs to $\R_+^*\subset \wR$.
It remains to note that 
$$\mathcal F_f=\cup_{\rho \in \mathfrak m_+} \tilde F_{\tilde f}(\rho,t).$$
\end{proof}

\begin{rmk}\label{rmk-cor} We have deduced  Corollary \ref{cor-wctriv} from Lemma \ref{lem-triv} by specifying for $\alpha$ the value $t\in (0,\tilde v(\tilde r_0)]_{\wR}$. In the proof of Theorem \ref{thm-hom} below, we will use the same result for a real value for $\alpha$, so belonging to $(0,v(r_0)]$.
\end{rmk}

\subsection{Comparison of homology groups}

The goal of this section is to prove that the natural maps from $\tilde H^w_{*}(\mathcal F_f)$ to $\tilde H_*(\mathcal F_f)$ and from $\tilde H_*(\mathcal F_f)$ to $H_*(F_f(a,r))$ are isomorphisms, for $a\in \R_+^*$ and $r\in \R_+^*$ small enough.

\begin{thm}\label{thm-hom} Let $f$ be a polynomial function on $\R^n$ vanishing at $0$.
Then there exist natural isomorphisms 
$$\tilde H^w_{*}(\mathcal F_f)\to\tilde H_*(\mathcal F_f)\to H_*(F_f(r,a))$$
with $0<a\ll r\ll 1$ small enough.
\end{thm}

First, we reduce the problem to the closed and bounded case by the following lemma.

\begin{lemma}\label{lem-homeq}
Let $\mathcal M$ be a local $\tilde\R$-Nash manifold possibly with corners, and $\mathcal N$ be a closed and bounded $\tilde\R$-Nash manifold possibly with corners. 
Let $\phi:\mathcal N\to\tilde\R$ be a positive $\tilde\R$-Nash function such that $\phi^{-1}(\mathfrak m_+)=\mathcal M$ and assume that the map 
$$\pi:\bigcup_{\rho\in\tilde\R}\{\gamma\in \mathcal N:\phi(\gamma)\le\rho\}\times\{\rho\}\longrightarrow \tilde\R$$ 
defined by $(\gamma,\rho)\to\rho$
is weakly continuously $\tilde\R$-Nash trivial over $[\rho_0,\,\tilde r_0]_{\tilde\R}$ for some $\rho_0\in\mathfrak m_+$ and $r_0\in\R_+^*$, via a trivialisation already defined over $\R$.

Then the inclusions $\mathcal M\to \phi^{-1}((0,\,\tilde r_0]_{\tilde\R})$ and $\phi^{-1}((0,\,\rho_0]_{\tilde\R})\to \mathcal M$ are local $\tilde\R$-semialgebraic weak homotopy equivalences.
\end{lemma}

\begin{proof} We are going to construct a relevant $\wR$-semialgebraic weak homotopy. Let denote by $\mathcal N_{\rho_0}$ and $\mathcal N_{\tilde r_0}$ respectively the $\wR$-Nash manifolds $\phi^{-1}((0,\,\rho_0]_{\tilde\R})$ and $\phi^{-1}((0,\,\tilde r_0]_{\tilde\R})$. By assumption, there exist a weakly continuous $\wR$-Nash $\wR$-diffeomorphism 
$$\psi:\mathcal N_{\tilde r_0}\times [\rho_0,\tilde r_0]_{\wR} \longrightarrow \bigcup_{\rho\in [\rho_0,\tilde r_0]_{\wR}}\{\gamma\in \mathcal N:\phi(\gamma)\le\rho\}\times\{\rho\}$$
such that $\pi \circ \psi$ is equal to the projection map
$$\mathcal N_{\tilde r_0} \times [\rho_0,\tilde r_0]_{\wR} \longrightarrow [\rho_0,\tilde r_0]_{\wR}.$$
We can assume moreover that $\psi$ is the identity map on restriction to $\mathcal N_{\tilde r_0} \times \{\tilde r_0\}$. Let denote by $\nu$ the projection
$$\nu:\bigcup_{\rho\in\tilde\R}\{\gamma\in \mathcal N:\phi(\gamma)\le\rho\}\times\{\rho\}\longrightarrow \wR^n$$ 
defined by $(\gamma,\rho)\to\gamma$. Then we construct a $\wR$-semialgebraic weak homotopy $\theta_{\lambda}: \mathcal N_{\tilde r_0} \longrightarrow \mathcal N_{\tilde r_0}$ by defining
$$\theta_{\lambda}(\gamma)=\nu \circ \psi \big(\gamma, \lambda(\rho_0-\tilde r_0)+\tilde r_0 \big)$$
for $\lambda \in [0,1]_{\wR}$ and $\gamma \in \mathcal N_{\tilde r_0}$. In particular $\theta_0$ is the identity map on $\mathcal N_{\tilde r_0}$, the image of $\mathcal N_{\tilde r_0}$ under $\theta_1$ is equal to $\mathcal N_{\rho_0}$ and moreover $\theta_{\lambda}(\mathcal M)\subset \mathcal M$ for any $\lambda \in [0,1]_{\wR}$ since the trivialisation $\psi$ comes from a trivialisation defined over $\R$.
\end{proof}

The following result is an immediate consequence of Lemma \ref{lem-homeq} and Lemma \ref{lem-wcc}.

\begin{cor}\label{cor-iso} Under the assumptions of Lemma \ref{lem-homeq}, the inclusions $\mathcal M\to \phi^{-1}((0,\,\tilde r_0]_{\tilde\R})$ and $\phi^{-1}((0,\,\rho_0]_{\tilde\R})\to \mathcal M$ are $\tilde\R$-semialgebraic $\tilde\R$-homotopy equivalences.
Therefore, the maps 
$$\tilde H_*(\phi^{-1}((0,\,\rho_0]_{\tilde\R}))\to\tilde H_*(\mathcal M)\to\tilde H_*(\phi^{-1}((0,\,\tilde r_0]_{\tilde\R}))$$ 
and 
$$\tilde H^w_{*}(\phi^{-1}((0,\,\rho_0]_{\tilde\R}))\to\tilde H^w_{*}(\mathcal M)\to\tilde H_{w*}(\phi^{-1}((0,\,\tilde r_0]_{\tilde\R}))$$ 
are isomorphisms.
\end{cor}

\begin{rmk} The above proof shows also that the inclusions $\mathcal M \longrightarrow \phi^{-1}([0,\tilde r_0]_{\wR})$ and $\phi^{-1}([0,\rho_0/2]_{\wR})\longrightarrow \mathcal M$ are homotopy equivalences in the sense of product topology, and the maps
$$H_*(\phi^{-1}((0,\rho_0]_{\wR})) \longrightarrow H_*(\mathcal M) \longrightarrow H_*(\phi^{-1}((0,\tilde r_0]_{\wR}))$$
are isomorphisms. 
\end{rmk}

\begin{proof}[Proof of theorem \ref{thm-hom}] The set $\mathcal F_f$ is the local $\wR$-Nash manifold associated to the $\wR$-Nash manifold $\mathcal A_f$ (actually an $\wR$-algebraic set) and to the $\wR$-Nash function $\phi:\mathcal A_f \longrightarrow \wR$ defined by $\phi(\gamma)=|\gamma|$ for $\gamma\in \wR^n$, cf. Example \ref{ex-f}. Note moreover that for $\rho\in \wR_+^*$, we have $\tilde F_{\tilde f}(\rho,t)=\phi^{-1}((0,\rho]_{\wR})$.

Using Corollary \ref{cor-wctriv} combined with Corollary \ref{cor-iso}, we obtain therefore that the morphisms
$$\tilde H_*(\tilde F_{\tilde f}(\rho_0,t))\longrightarrow \tilde H_*(\mathcal F_f)\longrightarrow\tilde H_*(\tilde F_{\tilde f}(\tilde r_0,t))$$ 
and 
$$\tilde H^w_{*}(\tilde F_{\tilde f}(\rho_0,t))\longrightarrow\tilde H^w_{*}(\mathcal F_f)\longrightarrow\tilde H^w_{*}(\tilde F_{\tilde f}(\tilde r_0,t))$$ 
are isomorphisms. We concentrate on $\tilde F_{\tilde f}(\tilde r_0,t)$. As noted in Remark \ref{rmk-cor}, for a sufficiently small vale $a_0\in \R_+^*$, we have similarly natural isomorphisms
$$H_*(\tilde F_{\tilde f}(\tilde r_0,t)) \longrightarrow H_*(\tilde F_{\tilde f}(\tilde r_0,\tilde a_0))$$
and
$$H^w_*(\tilde F_{\tilde f}(\tilde r_0,t)) \longrightarrow H^w_*(\tilde F_{\tilde f}(\tilde r_0,\tilde a_0)).$$
We achieve the proof using Lemma \ref{lem-comp} combined with Proposition \ref{prop-comp}, since Lemma \ref{lem-comp} provides a natural isomorphism from $H_*(\tilde F_{\tilde f}(\tilde r_0,\tilde a_0))$ to $H_*( F_{f}(r_0, a_0))$, whereas Proposition \ref{prop-comp} gives a natural isomorphism between $H^w_*(\tilde F_{\tilde f}(\tilde r_0,\tilde a_0))$ and $H_*( F_{f}(r_0, a_0))$.
\end{proof}
\section{A characterisation of Milnor fibres}
The final goal of the paper is to give a characterisation of the Milnor fibre $F_f(r,a)$ of a polynomial function $f$ in terms of its associated set of Puiseux series $\mathcal F_f$. We propose a semialgebraic (together with a piecewise linear) characterisation in Theorem \ref{main}.(1), valid in any dimension. We propose also a Nash characterisation, for which we need to exclude some small dimension for topological reasons.

Recall that if $h$ is a semialgebraic homeomorphism between semialgebraic neighborhoods of 0 in $\R^n$, then $\tilde h$ is an $\wR$-semialgebraic $\wR$-homeomorphism between $\wR$-semialgebraic $\wR$-neighborhoods of 0 in $\wR^n$.

\begin{thm}\label{main}
Let $f$ and $g$ be polynomial functions on $\R^n$ vanishing at 0, and let $h$ be a semialgebraic homeomorphism between semialgebraic neighborhoods of 0 in $\R^n$.

\begin{enumerate}
\item If $\tilde h(\mathcal F_f)=\mathcal F_g$ then $F_f(r,a)$ and $F_g(r,a)$ are semialgebraically homeomorphic for $0<a\ll r \ll 1$.

In particular, $F_f(r,a)$ and $F_g(r,a)$ are piecewise linear homeomorphic, in the sense that $C^\infty$ semialgebraic triangulations of $F_f(r,a)$ and $F_g(r,a)$ are piecewise linear homeomorphic.

\item Assume moreover that $\tilde h_{|\mathcal F_f}$ is a weakly continuous homeomorphism onto $\mathcal F_g$. Then $F_f(r,a)$ and $F_g(r,a)$ are Nash diffeomorphic for $0<a\ll r \ll 1$, under the condition that $n$ is not equal to 5 or 6.
If $n\not=5$, then $\Int F_f(r,a)$ and $\Int F_g(r,a)$ are analytically diffeomorphic for $0<a\ll r \ll 1$.
\end{enumerate}
\end{thm}

The proof of Theorem \ref{main}.(1) is based on piecewise linear topology \cite{PL}, and is exposed in section \ref{sect-PL}. The proof of Theorem \ref{main}(2) is more involved, and we use the theory of microbundles \cite{KS,Milnor} to discuss the Nash structures on the Milnor fibres. Note that Theorem \ref{main}.(2) is false without the additional condition that $\tilde h_{|\mathcal F_f}$ is a weakly continuous homeomorphism, as illustrated by the example of Kervaire's exotic sphere in section \ref{kerv}. We postpone the proof of this fact after the proof of Theorem \ref{main}.(1) since we make use of it to prove Proposition \ref{prop-cex}.

\subsection{Proof of the piecewise linear characterisation}\label{sect-PL}

The goal of the section is to understand to which extend the set of Puiseux series $\mathcal F_f$ determines the Milnor fibre $F_f(r,a)$ of a polynomial function $f$. Here is a first result is this direction.

\begin{lemma}\label{lem-deb} Let $f$ and $g$ be continuous semialgebraic function germs at $0$ in $\R^n$. If $\mathcal F_f=\mathcal F_g$ then $F_f(r,a)=F_g(r,a)$ for $0<a\ll r \ll 1$.
\end{lemma}

\begin{proof} It suffices to prove that $g=f$ on $\{g>0\}\cup\{f>0\}$.
Assuming it is not the case, there exists a continuous semialgebraic curve $\gamma: (\R,0)\longrightarrow (\R^n,0)$ along which $f$ is not equal to $g$ and either $f\circ \gamma (s)$ or $g\circ \gamma (s)$ is strictly positive for $s>0$ small enough. Assume for example that $g\circ \gamma (s)$ is strictly positive. We can suppose $g\circ \gamma (s)=s$ by changing the parameter $s$, so that $\gamma \in \mathcal F_g$. By assumption we obtain therefore $\gamma \in \mathcal F_f$, which contradicts that $f$ is not equal to $g$ along $\gamma$.
\end{proof}

Next result is the key argument in the proof of Theorem \ref{main}.(1). 

\begin{prop}\label{prop-key} Let $f$ be a continuous semialgebraic function on a compact semialgebraic subset $X$ of $\R^n$, with $0\in X$ and $f(0)=0$. Let $d_i$ be a non-negative continuous semialgebraic function on $X$ whose zero set is reduced to $\{0\}$, for $i\in \{1,2\}$.
Set 
$$
N_i(a,r)=\{x\in X:d_i(x)\le r,\,f(x)=a\}$$
for real numbers $a,r\in \R$, with $i\in \{1,2\}$.
Then $N_1(a,r)$ and $N_2(a,r)$ are semialgebraically homeomorphic for $0<a\ll r\ll 1$.
\end{prop}

We postpone the proof of Proposition \ref{prop-key} in order to show how we use it in the proof of Theorem \ref{main}.(1).

\begin{proof}[Proof of Theorem \ref{main}.(1)] The continuous semialgebraic functions $f$ and $g\circ h$ share the same set of Puiseux series $\mathcal F_f=\mathcal F_{g\circ h}$ since $\tilde h ^{-1}(\mathcal F_g)$, which coincides with $\mathcal F_f$ by assumption, is equal to $\mathcal F_{g\circ h}$. Therefore the sets $F_f(r,a)$ and $F_{g\circ h}(r,a)$ are equal for $0<a\ll r\ll 1$ by Lemma \ref{lem-deb}.

Define a distance function $d$ a compact neighbourdhood $X$ of $\R^n$ by $d(x)=|h^{-1}(x)|$. Then $F_g(r,a)$ is semialgebraically homeomorphic to 
$$N(a,r)=\{x\in X:d(x)\le r,\,g(x)=a\}$$
by Proposition \ref{prop-key}, for $0<a\ll r\ll 1$. This last set is carried to
$$\{y\in h^{-1}(X): |y|\le r,\,g\circ h (y)=a\}$$
by the semialgebraic homeomorphism $h^{-1}$, which is equal to $F_{g\circ h}(r,a)$ for $0<a\ll r\ll 1$. As a consequence $F_f(r,a)=F_{g\circ h}(r,a)$ is semialgebraically homeomorphic to $F_g(r,a)$.
\end{proof}

The proof of Proposition \ref{prop-key} is classical piecewise linear topology in the case all the data are piecewise linear. The most delicate part is to come back to this situation. We begin with a lemma in the piecewise linear case.

\begin{lemma}\label{lemPL} Let $f$ be a piecewise linear function on a compact polyhedron $X$ of $\R^n$, with $0\in X$ and $f(0)=0$. Let $d_i$ be a non-negative piecewise linear function on $X$ whose zero set is reduced to $\{0\}$, for $i\in \{1,2\}$. 
Then $N_1(a,r)$ and $N_2(a,r)$ are piecewise linear homeomorphic for $0<a\ll r\ll 1$.
\end{lemma}

\begin{proof} Let $K$ be a simplicial complex such that $X$ is the underlying polyhedron of $K$ and $f$ together with $d_1,d_2$ are simplicial on $K$, i.e. linear on each simplex in $K$.

Let $K'$ denote the barycentric subdivision of $K$, and choose $\epsilon$ so small that the set 
$$\{x\in X:0<d_i(x)\le\epsilon\},$$ 
for $i=1$ or $2$, does not contain any vertex in $K''$, the double barycentric subdivision of $K$.
Then by the uniqueness of regular neighborhoods (cf. Theorem 3.8 in \cite{PL}), there exist simplicial isomorphisms $\alpha_i$ from $K''$ to some simplicial subdivisions $K_i$ of $K$ such that
$$\alpha_i(|\st (0,K'')|)=\{x\in X:\phi_i(x)\le\epsilon\},$$
for $i=1$ or $2$, where the notation $\st (0,K'')$ denotes the star of $K''$ at $0$ (i.e. the simplicial complex obtained by taking all simplices adjacent to $0$) and $|\st(0,K'')|$ denotes its underlying polyhedron.
Explicitly, we define $\alpha_i$ to be the identity map on $\{0\}\cup(X-|\st(0,K'')|)$ and $\alpha_i(v)$ to be equal to $\phi_i(\epsilon)\cap l_v$ for each vertex $v$ in $\lk(0,K'')$, and we extend linearly $\alpha_i$ to each simplex in $K''$, where the notation $l_v$ denotes the segment with ends 0 and $v$. In particular, note that $\alpha^{-1}_1(f^{-1}(0))$ is equal to $\alpha^{-1}_2(f^{-1}(0))$ by linearity of $f$ on the simplices.

Hence, by replacing $f$ with $f\circ\alpha_i$ and $X$ with $|\st(0,K'')|$, we have reduce the problem to prove that, for $f_1$ and $f_2$ simplicial functions on $K$ such that $f_1^{-1}(0)=f_2^{-1}(0)$, the sets $f_1^{-1}(a)$ and $f_2^{-1}(a)$ are piecewise linear homeomorphic for small values of $a>0$. This statement also follows from the uniqueness of regular neighborhoods.
\end{proof}

In order to prove Proposition \ref{prop-key}, it is natural to hope to triangulate the maps $(f,d_i): X \to \R^2$, for $i=1,2$. However a global triangulation of a continuous semialgebraic map from a compact set to $\R^2$ is impossible in general. We overcome this difficulty using a weak triangulation (given in Lemma 5 in \cite{SII}) which will be sufficient to treat our local situation.

Before entering into the details of the proof, we recall the notion of cell complex (from \cite{PL}) that will be useful for the proof. A cell means a compact convex polyhedron in $\R^n$.
A cell is piecewise linear homeomorphic to a simplex.
A cell complex means a family of cells such that the boundaries of each cell is the union of some cells and the union of the interiors of the cells is a locally finite disjoint union.
Note also that, given a cell complex, the interiors of the elements form a cell complex in the sense of topology.
A cellular map $h:L_1\to L_2$ between cell complexes means a piecewise linear map $h:|L_1|\to|L_2|$ which linearly carries each element of $L_1$ to some element of $L_2$, where $|L_1|$ denotes the underlying polyhedron of $L_1$.
The Alexander trick (\cite{PL}, p. 37) is the statement which states that a piecewise linear homeomorphism between the boundaries of two cells is extended to a piecewise linear homeomorphism between the cells.
We can see how to apply the Alexander trick in \cite{PL} and \cite{SI}.

\begin{proof}[Proof of Proposition \ref{prop-key}] By the triangulation theorem of semialgebraic functions (Theorem 3.2 in \cite{SI}), we can suppose that $X$ is the underlying polyhedron of a simplicial complex $K$ and that $f$ is simplicial on $K$, with $f^{-1}(0)$ a union of simplices. Using Lemma \ref{lemPL}, we are reduced to prove that we can suppose that $d_1$ and $d_2$ are piecewise linear.

By Lemma 5 in \cite{SII}, there exist a semialgebraic homeomorphism $h_i$ of $X$
such that $h_i(\sigma)=\sigma$ for each $\sigma\in K$, and a neighbourhood $U$ of 0 in $f^{-1}(0)$ together with a compact polyhedral neighbourhood $V$ of $U-\{0\}$ in $X$ such that $(f,d_i)\circ h_i$ is piecewise linear on $V$ for $i=1$ and $2$.

Note that:
\begin{enumerate}
\item $V$ is not necessarily a neighbourhood of 0,
\item $d_i\circ h_i$ is piecewise linear only on $V$, but not necessarily on $f^{-1}([0,\,a])$,
\item $f\circ h_i$ is not necessarily piecewise linear on $f^{-1}([0,\,a])$ globally.
\end{enumerate}

The point $(2)$ is not annoying for the reduction to the case where $d_1$ and $d_2$ are piecewise linear because we are interested only in the set 
$$\{x\in X:d_i(x)\le\epsilon,\,f(x)=a\},$$
that is, we need the condition that $d_i$ is piecewise linear only on a neighbourhood of $\{x\in X:d_i(x)=\epsilon,\,f(x)=a\}$ in $X$.
However, the main difficulty consists in point $(3)$. In the sequel, we modify the semialgebraic homeomorphism  $h_i$ so that $f$ becomes equal to $f\circ h_i$ on $f^{-1}([0,\,a])$.

We compare first the functions $f$ and $f\circ h_i$ on $V$.
By subdividing $K$ we can assume that $V$ is the underlying polyhedron of some subcomplex of $K$, and that the $f\circ h_i$ are simplicial on $K|_V$ for $i=1,2$, although we may lose the property that $h_i(\sigma)=\sigma$ for $\sigma\in K$.
Set 
$$K|_{V\cap f^{-1}(a)}=\{\sigma\cap V\cap f^{-1}(a):\sigma\in K\}$$ 
for $a\in\R$. Then $K|_{V\cap f^{-1}(a)}$ is a cell complex (not necessarily simplicial).
Then, for $0<a\ll1\in\R$, there is a unique cellular isomorphism 
$$k_{i, a}:K|_{V\cap f^{-1}(a)}\to K|_{V\cap(f\circ\alpha_i)^{-1}(a)}$$ 
such that for each $\sigma\in K|_{V\cap f^{-1}(a)}$, the cells $\sigma$ and $k_{i, a}(\sigma)$ are included in some simplex in $K$ of dimension equal to $\dim\sigma+1$.
Hence, for some $a>0$, there is exist a piecewise linear homeomorphism 
$$k_i:V\cap f^{-1}([0,\,a])\to V\cap(f\circ h_i)^{-1}([0,\,a])$$ 
such that $f=f\circ h_i\circ k_i$ on $V\cap f^{-1}([0,\,a])$ and $k_i(\sigma)=\sigma$ for $\sigma\in K|_V$. Moreover, we can extend $h_i$ to a piecewise linear homeomorphism $h_i$ of $X$ so that $h_i(\sigma)=\sigma$ for $\sigma\in K$ by the Alexander trick. Note that $d_i\circ h_i\circ k_i$ continues to be piecewise linear on $V\cap f^{-1}([0,\,a])$.

So replacing $h_i$ with $h_i\circ k_i$, we have obtained that $f=f\circ h_i$ on $V\cap f^{-1}([0,\,a])$. 
Finally, we are going to modify $h_i$ outside of $V\cap f^{-1}([0,\,a])$ so that $f=f\circ h_i$ on $f^{-1}([0,\,a])$.
Set 
$$L=K|_{f^{-1}([0,\,a])}=\{\sigma\cap f^{-1}([0,\,a]):\sigma\in K\}$$ 
and consider the restriction $h'_i={h_i}_{|{V\cap|L|}}$.
Then $L$ is a cellular decomposition of $f^{-1}([0,\,a])$, the set $V\cap|L|$ is the underlying polyhedron of some subcomplex $L'$ of $L$ such that $h_i'(\sigma)=\sigma$ for $\sigma\in L'$, and $f=f\circ h_i'$ on $|L'|$.
Hence by the Alexander trick again, we can extend $h_i'$ to a semialgebraic homeomorphism $h_i''$ of $|L|$ so that $f=f\circ h_i''$ on $|L|$.
\end{proof}

\begin{rmk}\label{rmk-add}\begin{flushleft}
\end{flushleft}
\begin{enumerate}
\item As a consequence of the proof above, we can refine Theorem \ref{main}.(1) as follows. Choose $0<a_1\ll r_1\ll 1$ and $0<a_2\ll r_2\ll 1$ so that 
$$h(F_f(r_1,a_1)) \subset F_g(r_2,a_2).$$
Then
$$h_{|F_f(r_1,a_1)} :F_f(r_1,a_1) \to F_g(r_2,a_2)$$
is semialgebraically isotopic to a semialgebraic homeomorphism onto $F_g(r_2,a_2)$. We will use this refined version in the proof of Theorem \ref{main}.(2).
\item In the same spirit of the proof of Theorem \ref{main}.(1), we can prove that if $f$ and $g$ are polynomial functions on $\R^n$ vanishing at 0, if $h$ is a semialgebraic homeomorphism between semialgebraic neighbourhoods of 0 in $f^{-1}(0)\subset \R^n$, and if the restriction of $\tilde h$ to
$$\{\gamma \in {\tilde f^{-1}(0): \gamma (0)=0}\}$$
is a bijection onto
$$\{\gamma \in {\tilde g^{-1}(0): \gamma (0)=0}\}$$
then the sets $\{x\in f^{-1}(0): |x|\leq r\}$ and $\{x\in g^{-1}(0): |x|\leq r\}$ are semialgebraically homeomorphic, for $r>0$ small enough.
\end{enumerate}
\end{rmk}

\subsection{Kervaire's exotic sphere}\label{kerv}

Regard $\C$ as $\R^2$, and define polynomials $f$ and $g$ on $\C^6\times\R$ by 
$$f(z,x)=|z_1|^2+x^2$$
and
$$g(z,x)=|z_1^2+\cdots+z_5^2+z_6^3|^2+x^2,$$
for $(z,x)=(z_1,...,z_6,x)\in\C^6\times\R.$ Let $S$ denote the sphere in $\C^6\times\R$ with center 0 and with radius 1. Then set $g^{-1}(0)\cap S$ is the Kervaire's exotic sphere (cf. \cite{FMilnor} p. 72), i.e. $g^{-1}(0)\cap S$ is a topological sphere of dimension 9 (Theorem 8.5 and 9.1 in \cite{FMilnor}) which is not diffeomorphic to the standart 9-sphere. Note that $f^{-1}(0)\cap S$ is a standart sphere, also of dimension 9.

\begin{prop}\label{prop-cex}\begin{flushleft}
\end{flushleft}
\begin{enumerate}
\item There exists a semialgebraic homeomorphism $h$ of $\C^6\times\R$ such that $h(0)=0$ and $\tilde h(\mathcal F_f)=\mathcal F_g$.
\item $F_f(r,a)$ and $F_g(r,a)$ are not Nash diffeomorphic for $0<a\ll r \ll 1$.
\end{enumerate}
\end{prop}

\begin{proof}\begin{flushleft}
\end{flushleft}
\begin{enumerate}
\item We begin with constructing a semialgebraic homeomorphism $H$ of $S$ such that $$H(f^{-1}(0)\cap S)=g^{-1}(0)\cap S$$
and $f=g\circ H$ on a neighbourhood of $f^{-1}(0)\cap S$ in $S$. 

Regard $g^{-1}(0)\cap S$ and $f^{-1}(0)\cap S$ as piecewise linear manifolds by semialgebraic $C^1$ triangulations (cf. Proposition I.3.13 and Remark I.3.22 in \cite{S2} for semialgebraic $C^1$ triangulations).
Then these manifolds are piecewise linear homeomorphic to a standard piecewise linear sphere, since a topological sphere of dimension greater than or equal to 5 admits a unique piecewise linear manifold structure \cite{KS}.
Regard moreover the pairs $(S,g^{-1}(0)\cap S)$ and $(S,f^{-1}(0)\cap S)$ as piecewise linear manifold pairs (by semialgebraic $C^1$ triangulations).
Then they are unknotted piecewise linear sphere pairs by the Zeeman's unknotting theorem since 
$$\dim S-\dim f^{-1}(0)\cap S=\dim S-\dim g^{-1}(0)\cap S=3$$
(see Theorem 7.1 in \cite{PL}).
Hence both of them are standard piecewise linear sphere pairs, and therefore there exists a semialgebraic homeomorphism $H$ such that $H(f^{-1}(0)\cap S)=g^{-1}(0)\cap S$.

Moreover, we can modify $H$ so that the additional condition that $f=g\circ H$ on a neighbourhood of $f^{-1}(0)\cap S$ in $S$ holds, as we did in the proof of Proposition \ref{prop-key}.

Let $y=(z,x)\in S$. Define two semialgebraic maps 
$$l_y: [0,+\infty) \to \C^6\times\R,~~~ s\mapsto sy$$ 
and 
$$c_y: [0,+\infty) \to \C^6\times\R,~~~s\mapsto (s^{1/2}z_{1},...,s^{1/2}z_{5},s^{1/3}z_{6},s x)$$
so that the image $L_y$ of $l_y$ is the half line passing through $y\in S$, and the image $C_y$ of $c_y$ is a semialgebraic curve with origin $0\in \C^6\times\R$. Note moreover that
$$f\circ l_y(s)=s^2f(y)~~~\textrm{~~~and~~~}~~~ g\circ c_y(s)=s^2g(y)$$
for $s\in [0,+\infty)$, and that the functions $f\circ l_y$ and $g\circ c_y$ are strictly increasing and converging to $+\infty$ if $f(y)\neq 0$, respectively $g(y)\neq 0$. In particular, for $y\in S$ there exists an homeomorphism $H_y$ from $L_y$ onto $C_{H(y)}$ such that $H_y \circ l_y=c_{H(y)}$ if $f(y)=0$, and $g\circ H_y=f$ on $L_y$ otherwise.

We use these homeomorphisms to define a map 
$$h:C^6\times\R \to C^6\times\R$$
so that $h_{|L_y}$ coincides with $H_y:L_y \to C_{H(y)}$, for any $y\in S$. Then $h$ preserves the origin, and $h$ is well-defined $C^6\times\R\setminus \{0\}$ is the disjoint union of the half lines $L_y\setminus \{0\}$, for $y\in S$. Moreover $h$ is a  bijection because $C^6\times\R\setminus \{0\}$ is the disjoint union of thecurves $C_y\setminus \{0\}$, for $y\in S$, and it satisfies $f=g\circ h$ by construction.
Finally $h$ is clearly a semialgebraic map continuous at $S\setminus f^{-1}(0)$, and the continuity of $h$ at $S\cap f^{-1}(0)$ follows from the condition that $f=g\circ H$ on a neighbourhood of $f^{-1}(0)\cap S$ in $S$. As a consequence $h$ is a semialgebraic homeomorphism such that $\tilde h(\mathcal F)=\mathcal F_g$, as required.

\vskip 5mm
\item The set $g^{-1}(0)\cap S$ is the Kervaire's exotic sphere (cf. \cite{FMilnor} p. 72), whereas $f^{-1}(0)\cap S$ is a standart sphere. As a consequence the sets $F_f(r,a)$ and $F_g(r,a)$, whose boundary is diffeomorphic to $f^{-1}(0)\cap S$ and $g^{-1}(0)\cap S$ respectively, cannot be diffeomorphic.
\end{enumerate}
\end{proof}

\section{Microbundles and the Nash characterisation}

We begin with recalling some notions about microbundles, as introduced by J. Milnor in \cite{Milnor}. Then we use the notion of concordance of microbundles to deduce the required isomorphisms, using \cite{KS}.

\subsection{Microbundles}

A {\it microbundle} of rank $n\in \mathbb N$ is a diagram 
$$B\overset{i}{\longrightarrow}E\overset{j}{\longrightarrow}B$$ 
where $B$ and $E$ are topological spaces, such that the composition $j\circ i$ is the identity map, and for each $b\in B$ there exist an open neighbourhood $U$ of $b$, an open neighbourhood $V$ of $i(b)$ and a homeomorphism $h:V\to U\times\R^n$, with $i(U)\subset V$ and $j(V)\subset U$, such that the map $h\circ i|_U:U\to U\times\R^n$ coincides with the map $x\to(x,0)$ and the map $p_1\circ h:V\to U$ coincides with $j|_V$, where $p_1$ denotes the projection $U\times\R^n\to U$ onto the first coordinate.

A {\it smooth microbundle} is a microbundle such that $B$ and $E$ are $C^\infty$ manifolds possibly with boundary, such that $i$ and $j$ are of class $C^\infty$ and $h$ is a diffeomorphism.
A {\it smooth structure} on a topological manifold possibly with boundary is a $C^\infty$-equivalence class of atlases on it. A smooth structure on a microbundle is the data of smooth structures on $B$ and $E$ such that the microbundle becomes a smooth microbundle.

Let $\mathfrak b_k$ denote a microbundle $B\overset{i_k}{\longrightarrow}E_k\overset{j_k}{\longrightarrow}B$, for $k\in \{1,2\}$.
We say that $\mathfrak b_1$ is {\it isomorphic} to $\mathfrak b_2$ if there exist open neighbourhoods $U_1$ of $i_1(B)$ in $E_1$ and $U_2$ of $i_2(B)$ in $E_2$ together with a homeomorphism $h:U_1\to U_2$ such that the following diagram is commutative.
$$\xymatrix{
 B\ar[r]^{i_1}\ar[d]_{\id} & U_1\ar[r]^{\ j_1|_{U_1}\ \,}\ar[d]_{h} & B \ar[d]_{\id}\\
B \ar[r]^{i_2} & U_2\ar[r]^{\ j_2|_{U_2}\ \,} & B \\
}$$
If $h$ is, moreover, an inclusion, we say $\mathfrak b_1$ is micro-identical to $\mathfrak b_2$.

  An important example is the {\it tangent microbundle} $\frak t M$ of a topological manifold $M$ possibly with boundary, defined by
$$M\stackrel{D}\to M\times M\stackrel{p_1}\to M$$
where $D$ is the diagonal map (and $p_1$ still denotes the projection onto the first factor).
If $M$ is a $C^\infty$ manifold possibly with boundary, we regard the tangent vector bundle $p:TM\to M$, denoted by $\frak T M$, as a microbundle $M\to T M\to M$ by defining the maps $M\to T M$ and $T M\to M$ to be respectively the zero cross-section and the projection.
Then $\frak T M$ is isomorphic to $\frak t M$ (Theorem 2.2 in \cite{Milnor}.
We fix such an isomorphism for each $M$.

Let $(\frak t M)\times[0,\,1]$ denote the microbundle 
$$M\times[0,\,1]\stackrel{D\times\id}\longrightarrow M\times M\times[0,\,1]\stackrel{p_1\times\id}\longrightarrow M\times[0,\,1].$$
We define similarly a microbundle bundle $\mathfrak b\times[0,\,1]$ for any microbundle $\mathfrak b$.

In order to introduce a microbundle map, let us recall what a vector bundle map is. Note that any vector bundle over a topological space is a microbundle.
Let $\mathfrak b_1$ and $\mathfrak b_2$ be vector bundles $E_1\overset{j_1}{\longrightarrow}B_1$ and $E_2\overset{j_2}{\longrightarrow}B_2$, respectively, and let $g:B_1\to B_2$ be a continuous map.
A {\it vector bundle map} $\mathfrak b_1\to\mathfrak b_2$ covering $g$ is the following commutative diagram
$$\xymatrix{
 E_1\ar[r]^{G}\ar[d]_{j_1} & E_2\ar[d]_{j_2}\\
B_1\ar[r]^{g} & B_2 \\
}$$
such that $G$ is a continuous map and for each $x\in B_1$, $G|_{{j_1}^{-1}(x)}$ is a linear morphism onto ${j_2}^{-1}_2(g(x))$.
In a similar way, we define a microbundle map as follows.
Let $\mathfrak b_k$ be a microbundle $B_k\overset{i_k}{\longrightarrow}E_k\overset{j_k}{\longrightarrow}B_k$, for $k\in\{1,2\}$, and let $g:B_1\to B_2$ be a continuous map.
A {\it microbundle map} $\mathfrak b_1\to\mathfrak b_2$ covering $g$ is the following commutative diagram
$$\xymatrix{
 B_1\ar[r]^{i_1}\ar[d]_{g} & U_1\ar[r]^{\ j_1|_{U_1}\ \,}\ar[d]_{G} & B_1 \ar[d]_{g}\\
B_2\ar[r]^{i_2} & E_2\ar[r]^{\ j_2\ \,} & B_2 \\
}$$
such that $G|_{U_1\cap j^{-1}_1(b_1)}$ is an open embedding into $j^{-1}_2(g(b_1))$ for each $b_1\in B_1$, where $U_1$ is an open neighbourhood of $i_1(B_1)$ in $E_1$ and $G$ is a continuous map.
For a homeomorphism $g:M_1\to M_2$ between topological manifolds possibly with boundary, let {\boldmath$g$}$:\frak t M_1\to\frak t M_2$ denote the microbundle map covering $g$ defined by $U_1=M_1\times M_1$ and $G(x,y)=(g(x),g(y))$.

Given a microbundle $\mathfrak b$: $B\overset{i}{\longrightarrow}E\overset{j}{\longrightarrow}B,$
a topological space $A$ and a continuous map $g:A\to B$, we define the {\it induced microbundle} $g^*\mathfrak b$ by the diagram 
$$A\to\{(a,e)\in A\times E:g(a)=j(e)\}\to A$$
where the maps are defined respectively by $a\mapsto(a,i\circ g(a))$ and $(a,e)\mapsto a$.

\subsection{Concordance}

A {\it concordance} between smooth structures $\mathfrak b_0$ and $\mathfrak b_1$ on a microbundle $\mathfrak b$ is a smooth structure $\mathfrak A$ on $\mathfrak b\times[0,\,1]$ such that the restriction $\mathfrak A|_{B\times\{k\}}$ is micro-identical to $\mathfrak b_k$, for $k\in \{0,1\}$.
A {\it stable smooth structure} on $\mathfrak b$ means a smooth structure on a microbundle $B\to E\times\R^l\to B$ for some $l\in\N$, where the maps are defined by $b\mapsto(i(b),0)$ and $(e,x)\mapsto e$.
We naturally define a {\it stable concordance} between stable smooth structures.

A {\it concordance} between two smooth structures on a topological manifold possibly with boundary $M$ is a smooth structure $\mathfrak A$ on $M\times[0,\,1]$ such that the restriction $\mathfrak A|_{M\times\{0\}}$ coincides with the first smooth structure, and $\mathfrak A|_{M\times\{1\}}$ coincides with the second.

The homotopy theorem for microbundles (cf. \S 3 in \cite{Milnor}) is a useful tool to produce isomorphisms between microbundles. We recall its statement since we will use it in the sequel. Let $A$ and $B$ be topological spaces, $\mathfrak b$: $B\overset{i}{\longrightarrow}E\overset{j}{\longrightarrow}B$ be a microbundle and $f$ and $g$ be continuous maps from $A$ to $B$.
Assume $A$ is paracompact and $f$ and $g$ are homotopic. Then the homotopy theorem for microbundles states that $f^*\mathfrak b$ and $g^*\mathfrak b$ are isomorphic.

The following proposition gives a criteria for two smooth structures on a topological manifold to be diffeomorphic.

\begin{prop}\label{prop-con}
Let $M$ be a topological manifold possibly with boundary of dimension different from $4$ and $5$ if $\partial M\not=\emptyset$, or of dimension different from $4$ if $\partial M=\emptyset$. Let $M_0$ and $M_1$ be $C^\infty$ manifolds possibly with boundary which are homeomorphic to $M$ via $h_k:M\to M_k$, for $k\in\{0,1\}$.
Assume that there exist a vector bundle $\mathfrak b:V\to M\times[0,\,1]$ and a microbundle isomorphism $H:\mathfrak b\to(\frak t M)\times[0,\,1]$ such that 
$${\bold h}_k\circ(H|_{M\times\{k\}}):\mathfrak b|_{M\times\{k\}}\to\frak T M_k$$ 
is a vector bundle map covering $h_k$, for $k\in \{0,1\}$.
Then $M_0$ and $M_1$ are $C^\infty$ diffeomorphic.
\end{prop}

\begin{proof}
Consider the case $\partial M=\emptyset$.
The induced vector bundles $h^*_0(\frak T M_0)$ and $h^*_1(\frak T M_1)$ over $M$ give two smooth structures on $\frak t M$.
By our assumption, they are concordant.

These smooth structures can also be defined as follows. Assume $M$ be embedded in $\R^n$, and let $r:N\to M$ be a retraction of an open neighbourhood of $M$ in $\R^n$.
Let $r_k:N\to M$, for $k\in \{0,1\}$, be continuous maps homotopic to $r:N\to M$ so that $h_k\circ r_k$ are smooth.
Then $h_k^*(\frak T M_k)$ is micro-identical to $(h_k\circ r_k)^*(\frak T M_k)|_M$, for $k\in \{0,1\}$.
Apply the pull-back rule (see p.~166 in \cite{KS}).
The structures $(h_k\circ r_k)^*(\frak T M_k)|_M$ are the smooth structures on $\frak t M$ endowed from the smooth structures on $M$ given by $h_k:M\to M_k$, and the smooth structures $(h_k\circ r_k)^*(\frak T M_k)|_M$ are also the images of the smooth structures on $M$ under the pull-back rule.
Then Theorem 4.3 in Essay IV of \cite{KS} states that the two smooth structures on $M$ given by $M_0$ and $M_1$ are concordant. As a consequence, $M_0$ and $M_1$ are diffeomorphic by the Concordance Implies Isotopy Theorem 4.1 in Essay I, ibid.
Note moreover that we can choose the diffeomorphism to be isotopic to $h_1\circ h_0^{-1}$.

If $\partial M\not=\emptyset$, we can repeat the same arguments as above for $\partial M_k$, with $k\in \{0,1\}$.
Then we see that $\big(h^*_0(\frak T M_0)\big)|_{\partial M}$ and $\big(h^*_1(\frak T M_1)\big)|_{\partial M}$ are stably concordant, and we can apply Theorems 4.3 and 4.1 of \cite{KS} as above.
Hence the boundaries $\partial M_0$ and $\partial M_1$ are diffeomorphic, and moreover, modifying $h_k$, we can give a smooth structure to $\partial M$ so that $h_k|_{\partial M}:\partial M\to\partial M_k$ is a diffeomorphism, for $k\in \{0,1\}$.
By collaring, we can extend smooth structures on $\partial M$ to a neighbourhood $U$ of $\partial M$ in $M$ where $h_k$ is a $C^\infty$ embedding.
Then we can, once more, apply the above arguments to $\Int M_k$, for $k\in \{0,1\}$, because the relative versions of Theorems 4.3 and 4.1 in \cite{KS} hold.
Hence there exists a $C^\infty$ diffeomorphism $F:\Int M_0\to\Int M_1$ such that $F=h_1\circ h_0^{-1}$ on $h_0(U')\cap\Int M_0$ for a closed neighbourhood $U'$ of $\partial M$ in $M$ included in $U$.
Thus we see that $M_0$ and $M_1$ are diffeomorphic in the same way.
\end{proof}

\begin{cor}\label{cor-con}
Let $M_0$ and $M_1$ be $C^\infty$ manifolds possibly with boundary, of dimension different from $4$ or $5$ if $\partial M_k\not=\emptyset$ or of dimension different from $4$ if $\partial M_k=\emptyset$, with $k\in \{0,1\}$.
Assume that there exists an isotopy $g_s:M_0\to M_1$, with $0\le s\le1$, such that $g_0$ is a homeomorphism onto $M_1$ and $g_1$ is a $C^1$ embedding.
Then $M_0$ and $M_1$ are $C^\infty$ diffeomorphic.
\end{cor}

\begin{proof}
Set $M=M_0$, let $h_0: M\to M_0$ denote the identity map, and let $h_1$ denote the map $g_0:M\to M_1$. Let $\mathfrak b$ be the vector bundle $(\frak T M)\times[0,\,1]$, i.e.
$$(T M)\times[0,\,1]\to M\times[0,\,1].$$
We want to define a microbundle map $H:\mathfrak b\to(\frak t M)\times[0,\,1]$ so that the conditions in Proposition \ref{prop-con} are satisfied.
Consider the microbundle homotopy {\boldmath$g$}$_s:\frak t M_0\to\frak t M_1$ covering $g_s:M_0\to M_1$, with $0\le s\le 1$.
Then $g_0$ is bijective, but {\boldmath$g$}$_0$ is not necessarily a vector bundle map. Moreover $g_1$ is not necessarily bijective, however {\boldmath$g$}$_1$ is a vector bundle map.
Our goal is to modify {\boldmath$g$}$_s$ so that $g_s$ is bijective for any $s\in[0,\,1]$ and {\boldmath$g$}$_1$ remains a vector bundle map.

By the homotopy theorem for microbundles (or more precisely its proof in \S 6 of \cite{Milnor}), we have a microbundle homotopy $G_s:\frak t M_0\to\frak t M_1$, with $ 0\le s\le1$, covering $g_0$ such that $G_0=${\boldmath$g$}$_0$ and $G_1:\frak T M_0\to\frak T M_1$ is a vector bundle map, where we identify $\frak t M_k$ with $\frak T M_k$ by Theorem 2.2 in \cite{Milnor}, for $k\in \{0,1\}$.
Set
$$
H(x,s)=(\text{\boldmath$g$}^{-1}_0\circ G_s(x),s)\quad\text{for}\ (x,s)\in(T M)\times[0,\,1].
$$
Then $H$ is a microbundle map from $\mathfrak b$ to $(\frak t M)\times[0,\,1]$ such that {\boldmath$h$}$_0\circ(H|_{M\times\{0\}})$ is the identity map. Moreover 
$${\bold h}_1\circ(H|_{M\times\{1\}})={\bold h}_1\circ\,{\bold g}_0^{-1}\circ G_1=G_1$$
and hence {\boldmath$h$}$_0\circ(H|_{M\times\{0\}})$ and {\boldmath$h$}$_1\circ(\pi|_{M\times\{1\}})$ are vector bundle maps.
Therefore $\mathfrak b$ and $H$ satisfy the conditions in Proposition \ref{prop-con}, and therefore $M_0$ and $M_1$ are $C^\infty$ diffeomorphic.
\end{proof}

We are now equipped to handle the proof of Theorem \ref{main}.(2). We will prove this by applying Corollary \ref{cor-con}. To this aim, we will make use of Proposition \ref{prop-C1} in order to convert the weak continuity assumption on $\tilde h$ into a $\wR$-$C^1$ regularity first, and finally into a $C^1$ regularity over $\R$.


\begin{proof}[Proof of Theorem \ref{main}.(2)]
First, note that there exist a sufficiently small $r_0\in \R_+^*$ and a semialgebraic neighborhood $U$ of $(0,r_0]\times \{0\}$ in $(0,+\infty)^2$ such that the projection map 
$$p:\cup_{0<a<r\in \R} F_f(r,a)\times\{(r,a)\} \to \R^2$$ 
given by $(x,r,a)\mapsto (r,a)$ is Nash trivial over $U$, as in the proof of Lemma \ref{lem-triv}.
Hence there exist $l\in\N$ such that the projection map 
$$\cup_{r\in(0,r_0]} F_f(r,r^l)\times\{r\} \to (0,\,r_0]$$
given by $(x,r)\mapsto r$ is Nash trivial.
As a consequence, there exist a Nash diffeomorphism 
$$H_f: F_f(r_0,r_0^l)\times(0,\,r_0]\to \cup_{r\in(0,\,r_0]} F_f(r,r^l)$$ 
such that $H_f( F_f(r_0,r_0^l)  \times\{r\})=F_f(r,r^l)$ and $H_f(\cdot,r_0)=\id$.
In the same way, there exist a Nash diffeomorphism 
$$H_g: F_g(r_0^{l_1},r_0^l) \times (0,\,r_0] \to \cup_{r\in(0,\,r_0]} F_g(r^{l_1},r^l)$$ 
such that $H_g( F_g(r_0^{l_1},r_0^l)  \times\{r\})=F_g(r^{l_1},r^l)$ and $H_g(\cdot,r_0)=\id$, for some small $l_1\in\mathbf Q_+^*$, by enlarging $l$ and shrinking $r_0$ if necessary.
Here we can choose $l_1\in\mathbf Q_+^*$ so that 
$$h(\{x\in\R^n:|x|\le r\})\subset\{x\in\R^n:|x|\le r^{l_1}\}$$
for any $r\in[0,\,r_0]$ by semialgebraicity of $h$.

Define a semialgebraic isotopy 
$$g_s:  F_f(r_0,r_0^l)\to F_g(r_0^{l_1},r_0^l)$$
with $ 0\le s\le1/2$ (rather than $0\le s\le1$) using Remark \ref{rmk-add}.(1), so that $g_0$ is a homeomorphism onto $F_g(r_0^{l_1}, r_0^l)$ and $g_{1/2}(x)=h(x)$ for $x\in F_f(r_0,r_0^l)$.
Note that 
$$h\circ H_f(x,r_0)=H_g(g_{1/2}(x),r_0)$$
for $x\in F_f(r_0,r_0^l)$.
Next, extend $g_s$ to
$$g_s: F_f(r_0,r_0^l) \to F_g(r_0^{l_1},r_0^l),$$
using the continuous semialgebraic embeddings 
$$
h|_{F_f((2-2s)r_0,(2-2s)^lr_0^l)}: F_f((2-2s)r_0,(2-2s)^lr_0^l) \to F_g((2-2s)^{l_1}r_0^{l_1},(2-2s)^lr_0^l)$$
where $s\in[1/2,\,1)$, so that 
$$
h\circ H_f(x,(2-2s)r_0)=H_g(g_s(x),(2-2s)r_0)$$
for $s\in[1/2,\,1)$. Then $g_s$ is similar to an isotopy, except that the parameter $s$ moves only in $[0,1)$.

In order to achieve the proof, we want to apply Corollary \ref{cor-con} to $g_s$, with $0\le s\le s_0$, for some $s_0\in(0,\,1)$ close to 1. To this aim, it suffices to see that $g_s$ is a $C^1$ embedding for $s\in(0,\,1)$ sufficiently close to 1.

Extend the semialgebraic maps $g_s$ to $\wR$-semialgebraic maps 
$$\tilde g_{\lambda}:\tilde F_{\tilde f}(\tilde r_0,\tilde r_0^l) \to \tilde F_{\tilde g}(\tilde r_0^{l_1},\tilde r_0^l)$$
where $\lambda \in[0,\,1)_{\tilde\R}$.

Using the assumption that $\tilde h|_{\mathcal F_f}$ is a weakly continuous homeomorphism, together with Proposition \ref{prop-C1}, we obtain that $\tilde h$ is of class $\wR$-$C^1$. Then the restriction of $\tilde h$ to $\tilde F_{\tilde f}(\rho,\rho^l)$ is an $\tilde\R$-$C^1$ embedding into $\tilde F_{\tilde g}(\rho^{l_1},\rho^l)$ for $\rho=t^{1/l}$, since $\tilde F_{\tilde f}(t^{1/l},t)$ is included in $\mathcal F_f$. As a consequence $\tilde g_{\lambda}$ is an $\tilde\R$-$C^1$ embedding for $\lambda_0 \in\tilde\R_+^*$ such that $(2-2\lambda_0)\tilde r_0=t^{1/l}$.

Remark that $\lambda_0$ belongs to $[1/2,\,1)_{\tilde\R}$, so that the subset of $[1/2,\,1)_{\tilde\R}$ defined by
$$\{\lambda \in [1/2,\,1)_{\tilde\R}:  \textrm{~~$\tilde g_{\lambda}$ is not a $\tilde\R$-$C^1$ embedding}\}$$
is not empty. Moreover this set is the $\tilde\R$-extension of the semialgebraic subset of $[1/2,\,1)$ defined by
$$\{s\in [1/2,\,1): \textrm{~~$g_s$ is not a $C^1$ embedding}\},$$
because the former is described by the $\tilde\R$-extensions of the polynomials which describe the latter.
Therefore $g_s$ is a $C^1$ embedding for $s$ close to 1.
\end{proof}



\begin{thebibliography}{99}


\bibitem{BCR}  J. Bochnak, M. Coste, M.F. Roy, Real algebraic geometry, Springer, 1998

\bibitem{CF} G. Comte, G. Fichou, \textit{Grothendieck ring of semialgebraic formulas and motivic real Milnor fibre},  arXiv:1111.3181, Geometry \& Topology, to appear

\bibitem{CRS} M. Coste, J.M. Ruiz, M. Shiota, \textit{Approximation in compact Nash manifolds}, Amer. J. Math., 117 (1995) 905-927


\bibitem{CS2} M. Coste, M. Shiota, \textit{Thom's first isotopy lemma: a semialgebraic version, with uniform bound Real analytic and algebraic geometry}, Real Analytic and Algebraic Geometry, Walter de Gruyter, 83-101, 1995

\bibitem{DK} H. Delfs, M. Knebusch, \textit{On the homology of algebraic varieties over real closed fields} J. Reine Angew. Math. 335 (1982), 122-163

\bibitem{DL} J. Denef, F. Loeser, \textit{Germs of arcs on singular algebraic varieties and motivic integration}, Invent. Math. 135 (1999), no. 1, 201-232

\bibitem{FS} G. Fichou, M. Shiota, \textit{Continuous mappings between spaces of arcs}, Bull. Soc. Math. Fr., to appear

\bibitem{HK} E. Hrushovski, D. Kazhdan, \textit{ Integration in valued fields},  Algebraic geometry and number theory, 261–405, Progr. Math., 253, Birkhäuser Boston, Boston, MA, 2006

\bibitem{HL} E. Hrushovski, F. Loeser, \textit{Monodromy and the Lefschetz fixed point formula}, arXiv:1111.1954

\bibitem{KS} R. C. Kirby, L. C. Siebenmann, Foundational essays on topological manifolds, smoothings, and triangulations, Annals of Mathematics Studies, Princeton, 1977

\bibitem{MCP} C. McCrory, A. Parusi\'nski, \textit{Complex monodromy and the topology of real algebraic sets}, Compositio Math. 106 (1997), no. 2, 211-233

\bibitem{MMS} D. Macpherson, D. Marker, C. Steinhorn, \textit{Weakly o-minimal structures and real closed fields}, Trans. Amer. Math. Soc. 352 (2000), no. 12, 5435-5483

\bibitem{Milnor} J. Milnor, \textit{Microbundles I}, Topology 3 (1964) suppl. 1, 53-80 

\bibitem{FMilnor} J. Milnor, \textit{Singular points of complex hypersurfaces}, Ann. of Math. study 61, Princeton University Press, 1968


\bibitem{NS} J. Nicaise, J. Sebag, \textit{Motivic Serre invariants, ramification, and the analytic Milnor fiber}, Invent. Math. 168 (2007), no. 1, 133-173

\bibitem{PL} C. P. Rourke, B. J. Sanderson, Introduction to Piecewise-Linear Topology, Springer, 1982

\bibitem{Shiota} M. Shiota, Nash manifolds, Lect. Notes in Math. 1269, Springer-Verlag, 1987

\bibitem{S2} M. Shiota, \textit{Geometry of subanalytic and semialgebraic sets}, Birk\-h\"auser Progress in Math., 150, 1997

\bibitem{SI} M. Shiota, \textit{O-minimal Hauptvermutung for polyhedra I},  \textit{Invent. Math.} 196 (2014), no. 1, 163-232

\bibitem{SII} M. Shiota, \textit{O-minimal Hauptvermutung for polyhedra II}, arxiv:1310.5489

\bibitem{Y} Y. Yin, \textit{Additive invariants in o-minimal valued fields}, arXiv:1307.0224

\end{thebibliography}
\enddocument